\documentclass[a4paper]{amsart}
\usepackage[T1]{fontenc}
\usepackage{lmodern}
\usepackage[utf8]{inputenc}
\usepackage{amssymb}
\usepackage[all]{xy}
\usepackage{enumitem}
\usepackage{mathrsfs,mathtools}
\usepackage{nicefrac}

\usepackage[pdftitle={Bivariant K-theory via correspondences},
  pdfauthor={Heath Emerson and Ralf Meyer},
  pdfsubject={Mathematics; MSC 19K35, 46L80}
]{hyperref}
\newcommand*{\MRref}[2]{ \href{http://www.ams.org/mathscinet-getitem?mr=#1}{MR \textbf{#1}}}
\newcommand*{\arxiv}[1]{\href{http://www.arxiv.org/abs/#1}{arXiv: #1}}
\newcommand*{\doi}[2]{\href{http://dx.doi.org/#1}{#1}}

\usepackage[lite]{amsrefs}
\usepackage{microtype}
\numberwithin{equation}{section}
\theoremstyle{plain}
\newtheorem{theorem}[equation]{Theorem}
\newtheorem{lemma}[equation]{Lemma}
\newtheorem{proposition}[equation]{Proposition}
\newtheorem{corollary}[equation]{Corollary}
\theoremstyle{definition}
\newtheorem{definition}[equation]{Definition}

\theoremstyle{remark}
\newtheorem{remark}[equation]{Remark}
\newtheorem{example}[equation]{Example}

\DeclareMathOperator{\KK}{KK}

\DeclareMathOperator{\K}{K}
\DeclareMathOperator{\KO}{KO}
\DeclareMathOperator{\RK}{RK}
\DeclareMathOperator{\vbK}{VK}

\newcommand*{\GKK}{\widehat{\textsc{kk}}}
\newcommand*{\Coh}{\textup F}
\newcommand*{\Bic}{\hat{\textsc f}}

\newcommand*{\flip}{\textup{flip}}
\newcommand*{\PD}{\textup{PD}}
\newcommand*{\SPD}{\textup{PD}_2}

\newcommand*{\bd}{\partial}

\newcommand*{\Id}{\textup{Id}}

\newcommand*{\C}{\mathbb C}
\newcommand*{\Z}{\mathbb Z}

\newcommand*{\Sphere}{\mathbb S}
\newcommand*{\Disk}{\mathbb D}
\newcommand*{\N}{\mathbb N}
\newcommand*{\R}{\mathbb R}

\newcommand*{\Hils}{\mathcal H}
\newcommand*{\CONT}{\textup C}
\newcommand*{\EG}{\mathcal E}

\newcommand*{\Tvert}{\textup T}
\newcommand*{\Grd}{\mathcal G}
\newcommand*{\Base}{Z} 
\newcommand*{\base}{z} 
\newcommand*{\Tot}{X}  
\newcommand*{\tot}{x}  
\newcommand*{\Other}{Y}
\newcommand*{\Source}{X}
\newcommand*{\Target}{Y}
\newcommand*{\target}{y}
\newcommand*{\Third}{U}
\newcommand*{\third}{u}
\newcommand*{\Dual}{P}
\newcommand*{\dual}{p}
\newcommand*{\Midd}{M}
\newcommand*{\midd}{m}
\newcommand*{\Bord}{W}
\newcommand*{\bord}{w}
\newcommand*{\VB}{V}
\newcommand*{\vb}{v}
\newcommand*{\Triv}{E}
\newcommand*{\triv}{e}
\newcommand*{\Normal}{\textup N}

\newcommand*{\Kclass}{\xi}
\newcommand*{\Thom}{\tau}

\newcommand*{\NM}{\Phi}
\newcommand*{\Cor}{\Psi}
\newcommand*{\anchor}{\varrho}
\newcommand*{\mapl}{b}
\newcommand*{\mapr}{f}
\newcommand*{\pr}{\textup{pr}}

\newcommand*{\nb}{\nobreakdash}
\newcommand*{\Cst}{\textup C^*}

\newcommand*{\total}[1]{\lvert#1\rvert}
\newcommand*{\proj}[1]{\pi_{#1}}
\newcommand*{\bproj}[1]{\bar\pi_{#1}}
\newcommand*{\zers}[1]{\zeta_{#1}}
\newcommand*{\norm}[1]{\lVert#1\rVert}

\newcommand*{\defeq}{\mathrel{\vcentcolon=}}
\newcommand*{\pt}{\star}
\newcommand*{\blank}{\textup{\textvisiblespace}}

\newcommand*{\inpro}{\mathbin{\#}}

\newcommand*{\mono}{\rightarrowtail}
\newcommand*{\epi}{\twoheadrightarrow}
\newcommand*{\opem}{\hookrightarrow}
\newcommand*{\bordant}{\sim_\textup b}
\newcommand*{\sbordant}{\sim_{\textup{sb}}}
\newcommand*{\sequiv}{\sim_\textup s}

\begin{document}
\title{Bivariant K-theory via correspondences}

\author{Heath Emerson}
\email{hemerson@math.uvic.ca}
\address{Department of Mathematics and Statistics\\
  University of Victoria\\
  PO BOX 3045 STN CSC\\
  Victoria, B.C.\\
  Canada V8W 3P4}

\author{Ralf Meyer}
\email{rameyer@uni-math.gwdg.de}
\address{Mathematisches Institut and Courant Research Centre ``Higher Order Structures''\\
  Georg-August Universit\"at G\"ottingen\\
  Bunsenstra{\ss}e 3--5\\
  37073 G\"ottingen\\
  Germany}

\begin{abstract}
  We use correspondences to define a purely topological equivariant bivariant \(\K\)\nb-theory for spaces with a proper groupoid action.  Our notion of correspondence differs slightly from that of Connes and Skandalis.  Our construction uses no special features of equivariant \(\K\)\nb-theory.  To highlight this, we construct bivariant extensions for arbitrary equivariant multiplicative cohomology theories.

  We formulate necessary and sufficient conditions for certain duality isomorphisms in the topological bivariant \(\K\)\nb-theory and verify these conditions in some cases, including smooth manifolds with a smooth cocompact action of a Lie group.  One of these duality isomorphisms reduces bivariant \(\K\)\nb-theory to \(\K\)\nb-theory with support conditions.  Since similar duality isomorphisms exist in Kasparov theory, the topological and analytic bivariant \(\K\)\nb-theories agree if there is such a duality isomorphism.
\end{abstract}
\thanks{Heath Emerson was supported by a National Science and Engineering Council of Canada (NSERC) Discovery grant.  Ralf Meyer was supported by the German Research Foundation (Deutsche Forschungsgemeinschaft (DFG)) through the Institutional Strategy of the University of G\"ottingen.}
\maketitle

\section{Introduction}
\label{sec:intro}

Kasparov's bivariant \(\K\)\nb-theory is the main tool in non-commutative topology.  Its deep analytic properties are responsible for many applications of \(\Cst\)\nb-algebra methods in topology such as the Novikov conjecture or the Gromov--Lawson--Rosenberg conjecture.  But some of its applications -- such as the computation of equivariant Euler characteristics and Lefschetz invariants in~\cites{Emerson-Meyer:Euler, Emerson-Meyer:Equi_Lefschetz} -- should not require any difficult analysis and should therefore profit from a purely topological substitute for Kasparov's theory.  Our goal here is to construct such a theory in terms of correspondences.

Already in~1980, Paul Baum and Ronald Douglas~\cite{Baum-Douglas:K-homology} proposed a topological description of the \(\K\)\nb-homology \(\KK_*(\CONT_0(\Source),\C)\) of a space~\(\Source\), which was soon extended to the bivariant case by Alain Connes and Georges Skandalis~\cite{Connes-Skandalis:Indice_feuilletages}.  Equivariant generalisations with somewhat limited scope were considered in \cites{Baum-Block:Bicycles, Raven:Thesis}.  One might hope that a topological bivariant theory defined along these lines could be shown to agree with Kasparov's analytic theory \(\KK^\Grd_*\bigl(\CONT_0(\Source), \CONT_0(\Target)\bigr)\) under some finiteness assumptions.  But even in the case of non-equivariant \(\K\)\nb-homology, a complete proof appeared only recently in~\cite{Baum-Higson-Schick:Equivalence}.

The following problem creates new difficulties in the equivariant case.  Part of the data of a geometric cycle in the sense of Paul Baum is a vector bundle.  But equivariant vector bundles are sometimes in too short supply to generate equivariant \(\K\)\nb-theory.  Let \(\vbK^0_\Grd(\Tot)\) be the Grothendieck group of the additive category of \(\Grd\)\nb-equivariant complex vector bundles over a proper \(\Grd\)\nb-space~\(\Tot\).  The functor \((\Tot,A)\mapsto \vbK^0_\Grd(\Tot,A)\) for finite \(\Grd\)\nb-CW-pairs need not satisfy excision.

To get a reasonable cohomology theory, we need more general cycles as in Graeme Segal's definition of representable \(\K\)\nb-theory in~\cite{Segal:Fredholm_complexes}.  The \(\Grd\)\nb-equivariant representable \(\K\)\nb-theory \(\RK^*_\Grd(\Tot)\) for locally compact groupoids~\(\Grd\) and locally compact, proper \(\Grd\)\nb-spaces~\(\Tot\) is studied in~\cite{Emerson-Meyer:Equivariant_K}.  There are several equivalent definitions, using a variant of Kasparov theory, \(\K\)\nb-theory for projective limits of \(\Cst\)\nb-algebras, or equivariant families of Fredholm operators.  Furthermore, \cite{Emerson-Meyer:Equivariant_K} studies non-representable equivariant \(\K\)\nb-theory \(\K^*_\Grd(\Tot)\) and \(\K\)\nb-theory with \(\Target\)\nb-compact support \(\RK^*_{\Grd,\Target}(\Tot)\), where \(\Target\) and~\(\Tot\) are \(\Grd\)\nb-spaces with a \(\Grd\)\nb-map \(\Tot\to\Target\).  All three theories may be described by Fredholm-operator-valued maps~-- a reasonably satisfactory homotopy theoretic picture.

However, even after replacing \(\vbK^0_\Grd(\Tot)\) by \(\RK^0_\Grd(\Tot)\), equivariant vector bundles still play an important role in various arguments with correspondences.  First, the proof that the topological and analytic bivariant \(\K\)\nb-theories agree for smooth manifolds requires certain equivariant vector bundles, which only exist under additional technical assumptions.  Secondly, we need some equivariant \(\Grd\)\nb-vector bundles to compose correspondences.  If correspondences are defined as in~\cite{Connes-Skandalis:Longitudinal}, then composing them requires a transversality condition.  In the equivariant case, this can no longer be achieved by a perturbation argument.  A basic example is the pair of maps \(\{0\} \to \R\leftarrow \{0\}\) from a point to the plane sending the point to the origin.  This is equivariant with respect to the action of~\(\Z/2\) by a rotation around the origin. These two maps cannot be perturbed to be transverse to each other since the origin is the only fixed-point.  Baum and Block~\cite{Baum-Block:Bicycles} suggest how to compose such correspondences despite this.  This trick uses vector bundle modification and thus an ample supply of vector bundles.

Therefore, to get an elegant theory, we have modified two details in the definition of a correspondence.  Our changes to the definition have the nice side effect that our theory no longer uses any special features of \(\K\)\nb-theory and extends almost literally to any equivariant multiplicative cohomology theory.  We work in this general setting for conceptual reasons and in order to prepare for the construction of an equivariant bivariant Chern character.  In the non-equivariant case, Martin Jakob has described the homology and bivariant cohomology theories associated to a cohomology theory along similar lines in \cites{Jakob:Bordism_homology, Jakob:Bivariant}.

A \(\Grd\)\nb-equivariant \emph{correspondence} from~\(\Source\) to~\(\Target\) is a \(\Grd\)\nb-space~\(\Midd\) with \(\Grd\)\nb-equivariant continuous maps
\[
\Source \xleftarrow{\mapl} \Midd \xrightarrow{\mapr} \Target
\]
and with some equivariant \(\K\)\nb-theory datum~\(\Kclass\) on~\(\Midd\).  In~\cite{Connes-Skandalis:Longitudinal}, \(\mapl\) is proper, \(\mapr\) is smooth and \(\K\)\nb-oriented, and~\(\Kclass\) is a vector bundle over~\(\Midd\).

We do \emph{not} require~\(\mapl\) to be proper.  Instead, we let \(\Kclass\in \RK^*_{\Grd,\Source}(\Midd)\) be a \(\Grd\)\nb-equivariant \(\K\)\nb-theory class with \emph{\(\Source\)\nb-compact support}; thus \(\mapl\) and~\(\Kclass\) combine to an element of \(\KK^\Grd_* \bigl(\CONT_0(\Source), \CONT_0(\Midd)\bigr)\).  Roughly speaking, instead of requiring the fibres of~\(\mapl\) to be compact, we require that~\(\mapl\) restricts to a proper map on the support of~\(\Kclass\).

Furthermore, we let~\(\mapr\) be a \(\K\)\nb-oriented \emph{normally non-singular} map in the sense of~\cite{Emerson-Meyer:Normal_maps}.  Roughly speaking, these are maps together with a factorisation
\[
\xymatrix{
  \VB \ar[r]^-{\hat{f}}& \Triv \ar[d]^-{\proj{\Triv}}\\
  \Source \ar[u]^-{\zers{\VB}} \ar[r]^-{f}& \Target,
}
\]
where~\(\VB\) is a \(\Grd\)\nb-vector bundle over~\(\Source\) with zero section \(\zers{\VB}\colon \Source\to \VB\), \(\Triv\) is a \(\Grd\)\nb-vector bundle over~\(\Target\) with bundle projection \(\proj{\Triv}\colon \Triv \to \Target\), and \(\hat{f}\colon \VB\to \Triv\) is an open embedding.  The normally non-singular map is \(\K\)\nb-oriented if \(\VB\) and~\(\Triv\) are equivariantly \(\K\)\nb-oriented.  For technical reasons, we require~\(\Triv\) to be trivial, that is, pulled back from the object space of our groupoid, and~\(\VB\) to be subtrivial, that is, a direct summand in a trivial \(\Grd\)\nb-vector bundle.  The triviality of~\(\Triv\) is needed to define the composition of normally non-singular maps.  The subtriviality of~\(\VB\) is needed to bring correspondences into a standard form; otherwise the bundle projection~\(\proj{\VB}\) would not be the trace of a normally non-singular map.

Thom isomorphisms and functoriality suffice to construct purely topological wrong-way maps for normally non-singular maps.  In contrast, the construction of wrong-way elements for arbitrary smooth maps is analytical (see~\cite{Emerson-Meyer:Normal_maps}).  If we used smooth maps instead of normally non-singular maps, then correspondences would not in general describe equivariant \(\K\)\nb-theory correctly, and since we prove the equivalence of the topological and analytic bivariant \(\K\)\nb-theory by a reduction to \(\K\)\nb-theory, this would have a disastrous effect on the general framework. On the other hand, for many proper groupoids, any smooth map has an essentially unique normal factorisation, so that there is no difference between smooth normally non-singular maps and smooth maps.  But this requires some technical conditions on the groupoid, which then have to appear in all important theorems.  We use normally non-singular maps here to avoid such technical conditions.

Before we discuss correspondences further, we must discuss the kind of equivariance we allow.  Although we are mainly interested in the case of group actions on spaces, we develop our whole theory in the setting where~\(\Grd\) is a \emph{numerably proper} groupoid in the sense of~\cite{Emerson-Meyer:Normal_maps}.  Numerably proper groupoids combine Abels' numerably proper group actions (\cite{Abels:Universal}) with Haar systems.  If the groupoid~\(\Grd\) is not proper --~say an infinite discrete group~-- then we replace it by the groupoid \(\Grd\ltimes\EG\Grd\) for a universal proper \(\Grd\)\nb-space and pull back all \(\Grd\)\nb-spaces to \(\Grd\ltimes\EG\Grd\)-spaces.  This does not change \(\KK^\Grd_*\bigl(\CONT_0(\Source),\CONT_0(\Target)\bigr)\) if~\(\Grd\) acts properly (or amenably) on~\(\Source\) (see \cites{Meyer-Nest:BC, Emerson-Meyer:Dualities}).

Analysis plays no role in the construction of our bivariant cohomology theories.  Hence we do not need our spaces to be locally compact~-- paracompact Hausdorff is good enough.  For actions of numerably proper groupoids on paracompact Hausdorff spaces, pull-backs of equivariant vector bundles along equivariantly homotopic maps are isomorphic, equivariant vector bundles carry invariant inner products, and extensions of equivariant vector bundles split.

We removed the properness condition on~\(\mapl\) in order to simplify the construction of the intersection product.  In the usual approach, the intersection product of correspondences only works under a transversality assumption, which can be achieved by perturbing the maps involved.  This perturbation no longer works equivariantly.  As mentioned above, Paul Baum and Jonathan Block suggest in \cites{Baum-Block:Bicycles, Baum-Block:Excess} to use vector bundle modification to overcome this, but while this works well in certain situations we found it hard to formalise.  With our non-proper correspondences, we can bring any correspondence into a standard form for which transversality is automatic.

This involves our equivalence relation of Thom modification, which replaces the vector bundle modification of Paul Baum.  The difference is that we use the total space of the vector bundle instead of a sphere bundle.  This is possible because we allow non-proper correspondences.

Write the oriented normally non-singular map~\(\mapr\) from~\(\Midd\) to~\(\Target\) in a correspondence as a triple \((\VB,\Triv,\hat{\mapr})\) for oriented \(\Grd\)\nb-vector bundles \(\VB\) and~\(\Triv\) over~\(\Midd\) and~\(\Target\) and an open embedding \(\hat{\mapr}\colon \total{\VB}\opem \total{\Triv}\).  Thom modification along the vector bundle~\(\VB\) replaces the given correspondence by one that involves the total space of~\(\VB\) in the middle, and where~\(\mapr\) becomes the normally non-singular map \(\total{\VB}\subseteq \total{\Triv}\epi\Target\); such normally non-singular maps are also called \emph{special normally non-singular submersions}.  Thus any correspondence is equivalent to a special one, that is, one with a special normally non-singular submersion~\(\mapr\).  Notice that the map \(\total{\VB}\epi\Midd\to\Source\) is almost never proper.

Since \emph{any} map is transverse to a special normally non-singular submersion, it is easy to describe the intersection product for special correspondences and to check that it has the expected properties, including functoriality of the canonical map to Kasparov theory.  Since correspondences that appear in practice are usually not special, we define a notion of transversality for general correspondences and describe intersection products more directly in the transverse case.

The wrong-way functoriality for normally non-singular maps in~\cite{Emerson-Meyer:Normal_maps} provides a natural transformation
\begin{equation}
  \label{eq:intro_GKK_to_KK}
  \GKK_\Grd^*(\Source,\Target) \to
  \KK^\Grd_*\bigl( \CONT_0(\Source), \CONT_0(\Target) \bigr).
\end{equation}
The main result of this article is that~\eqref{eq:intro_GKK_to_KK} is an isomorphism if~\(\Grd\) is a proper locally compact groupoid with Haar system and~\(\Source\) admits a normally non-singular map to \(\Base\times[0,1)\) or to~\(\Base\), where~\(\Base\) is the object space of~\(\Grd\).  In the non-equivariant case, this assumption means that \(\Source\times\R^n\) carries a structure of smooth manifold with boundary for some \(n\in\N\).  In the equivariant case, the existence of such a normally non-singular map implies that there is a \(\Grd\)\nb-vector bundle~\(\Triv\) over~\(\Base\) such that \(\Source\times_\Base\total{\Triv}\) is a bundle of smooth manifolds with boundary over~\(\Base\), with a fibrewise smooth action of~\(\Grd\).  Conversely, such a smooth structure yields a normally non-singular map \(\Source\to\Base\) under a technical assumption about equivariant vector bundles.  The additional technical assumption holds, for instance, if \(\Grd= G\ltimes \Third\) for a discrete group~\(G\) and a finite-dimensional proper \(\Grd\)\nb-space~\(\Third\) with uniformly bounded isotropy groups, or if~\(\Grd\) is a compact group and~\(\Source\) is compact.

The proof that~\eqref{eq:intro_GKK_to_KK} is an isomorphism for such spaces~\(\Source\) is based on Poincar\'e duality.  If~\(\Source\) admits a normally non-singular map to \(\Base\times[0,1)\), then we describe a \(\Grd\)\nb-space~\(\Dual\) that is dual to~\(\Source\) in the sense that there are natural isomorphisms
\begin{align}\label{eq:intro_dual1}
  \GKK_{\Grd\ltimes\Source}^*(\Source\times_\Base\Third,
  \Source\times_\Base\Target) &\cong
  \GKK_\Grd^*(\Dual\times_\Base\Third,\Target),\\
  \label{eq:intro_dual2}
  \GKK_\Grd^*(\Source\times_\Base\Third,\Target) &\cong
  \GKK_{\Grd\ltimes\Source}^*(\Source\times_\Base \Third,
  \Dual\times_\Base\Target)
\end{align}
for any \(\Grd\)\nb-spaces \(\Target\) and~\(\Third\), and similarly for the analytic theory \(\KK\) instead of~\(\GKK\).  The corresponding duality in Kasparov theory is studied in~\cite{Emerson-Meyer:Dualities}, where sufficient and necessary conditions for it are established.  These criteria carry over almost literally to the topological version of Kasparov theory.

In particular, \eqref{eq:intro_dual2} for \(\Third=\Base\) identifies
\begin{align*}
  \GKK_\Grd^*(\Source,\Target) &\cong
  \GKK_{\Grd\ltimes\Source}^*(\Source,
  \Dual\times_\Base\Target),\\
  \KK^\Grd_*\bigl( \CONT_0(\Source), \CONT_0(\Target) \bigr)
  &\cong
  \KK^{\Grd\ltimes\Source}_*\bigl( \CONT_0(\Source),
  \CONT_0(\Dual\times_\Base\Target) \bigr)
  \cong \RK^*_{\Grd,\Source}(\Dual\times_\Base\Target),
\end{align*}
where \(\RK^*_{\Grd,\Source}(\Dual\times_\Base\Target)\) is the \(\Grd\)\nb-equivariant \(\K\)\nb-theory of \(\Dual\times_\Base\Target\) with \(\Source\)\nb-compact support.  The map in~\eqref{eq:intro_GKK_to_KK} is an isomorphism if \(\Source=\Base\) because our bivariant \(\K\)\nb-theory extends ordinary \(\K\)\nb-theory.  Hence~\eqref{eq:intro_GKK_to_KK} is an isomorphism whenever~\(\Source\) has a duality isomorphism~\eqref{eq:intro_dual2}.  Using the results in~\cite{Emerson-Meyer:Normal_maps}, this implies that~\eqref{eq:intro_GKK_to_KK} is invertible provided~\(\Source\) is a smooth \(\Grd\)\nb-manifold with boundary and some technical assumptions about \(\Grd\)\nb-vector bundles are satisfied.

We do not expect our topological bivariant \(\K\)\nb-theory to have good properties in the same generality in which it may be defined.  Equivariant Kasparov theory has good excision properties (long exact sequences) for proper actions in complete generality; in contrast, we would be surprised if the same were true for our topological theory.  We have not studied its excision properties, but it seems likely that excision requires some technical assumptions.  Correspondingly, we do not expect our theory to agree with Kasparov theory in all cases.


We would like to extend our thanks to Paul Baum for a number of interesting conversations on the subject of topological \(\KK\)-theory.

\section{Correspondences}
\label{sec:groupoids_actions}

We first define correspondences.  Then we define equivalence of correspondences, using the elementary relations of equivalence of normally non-singular map, bordism, and Thom modification.  Equivalence classes of \(\Coh\)\nb-oriented correspondences will be shown to form a group, which we denote by \(\Bic^*(\Source,\Target)\) or \(\Bic_\Grd^*(\Source,\Target)\) in the \(\Grd\)\nb-equivariant case.

The intersection product defining the composition of correspondences is only well-defined under a transversality condition (see~\cite{Connes-Skandalis:Longitudinal}).  We restrict attention to special correspondences at some point to rule out this problem.  The Thom modification allows us to replace any correspondence by one whose normally non-singular map is a special normally non-singular submersion, and this implies the transversality condition for all intersection products.  Thus we turn~\(\Bic^*\) into a \(\Z\)\nb-graded category.  We also define an exterior product that turns it into a \(\Z\)\nb-graded symmetric monoidal category.

Before we start with this, we briefly recall some prerequisites for our theory from \cites{Emerson-Meyer:Equivariant_K, Emerson-Meyer:Normal_maps}.

Throughout this article, all topological spaces, including all topological groupoids, are assumed to be paracompact and Hausdorff.  We shall use the notion of a numerably proper groupoid introduced in~\cite{Emerson-Meyer:Normal_maps}.  Equivariant vector bundles for numerably proper actions behave like non-equivariant vector bundles: equivariant sections or equivariant vector bundle morphisms extend from closed invariant subspaces, vector bundle extensions split equivariantly, and pull-backs along equivariantly homotopic maps are isomorphic.  An action of a locally compact groupoid with Haar system on a locally compact space is numerably proper if and only if the action is proper and the orbit space is paracompact.

As in~\cite{Emerson-Meyer:Normal_maps}, we write~\(\total{\VB}\) for the total space, \(\proj{\VB}\) for the bundle projection, and~\(\zers{\VB}\) for the zero section of a vector bundle~\(\VB\).  We reserve the arrows \(\mono\), \(\epi\), and~\(\opem\) for zero sections, vector bundle projections, and open embeddings, respectively.  A \(\Grd\)\nb-vector bundle is called \emph{trivial} if it is pulled back from the object space of~\(\Grd\), and \emph{subtrivial} if it is a direct summand in a trivial \(\Grd\)\nb-vector bundle.

Since our constructions use no special properties of \(\K\)\nb-theory, we mostly work with a general equivariant multiplicative cohomology theory~\(\Coh_\Grd\) as in~\cite{Emerson-Meyer:Normal_maps}.  Given~\(\Coh_\Grd\) and a \(\Grd\)\nb-space~\(\Tot\), the \(\Coh\)\nb-cohomology \(\Coh^*_{\Grd,\Tot}(\Other)\) of~\(\Other\) with \(\Tot\)\nb-compact support and the notion of an \(\Coh\)\nb-oriented \(\Grd\)\nb-vector bundle are defined in~\cite{Emerson-Meyer:Normal_maps}.  An \(\Coh\)\nb-oriented vector bundle~\(\VB\) over~\(\Other\) has a Thom isomorphism \(\Coh_\Tot^*(\Other)\cong \Coh_\Tot^*(\total{\VB})\).  Furthermore, \(\Coh^*_{\Grd,\Tot}\) is functorial for open embeddings.

The cohomology theory we are most interested in is (representable) equivariant \(\K\)\nb-theory.  The resulting \(\K\)\nb-theory with \(\Tot\)\nb-compact support agrees with the corresponding theory defined in~\cite{Emerson-Meyer:Equivariant_K}.  Actually, equivariant (representable) \(\K\)\nb-theory is defined in~\cite{Emerson-Meyer:Equivariant_K} only for locally compact proper \(\Grd\)\nb-spaces, so that we should impose such restrictions whenever we want to specialise to \(\K\)\nb-theory.

Normally non-singular maps are introduced in~\cite{Emerson-Meyer:Normal_maps}.  Since some details of this definition will become crucial here, we recall it:

\begin{definition}
  \label{def:normal_map}
  Let~\(\Grd\) be a numerably proper groupoid with object space~\(\Base\) and let \(\Source\) and~\(\Target\) be \(\Grd\)\nb-spaces.  Let~\(\Coh\) be a \(\Grd\)\nb-equivariant multiplicative cohomology theory.  An \(\Coh\)\nb-oriented \emph{normally non-singular \(\Grd\)\nb-map} from~\(\Source\) to~\(\Target\) consists of
  \begin{itemize}
  \item \(\VB\), an \(\Coh\)\nb-oriented subtrivial \(\Grd\)\nb-vector bundle over~\(\Source\);

  \item \(\Triv\), an \(\Coh\)\nb-oriented \(\Grd\)\nb-vector bundle over~\(\Base\);

  \item \(\hat{f}\colon \total{\VB} \opem \total{\Triv^\Target}\), an open embedding (that is, \(\hat{f}\) is a \(\Grd\)\nb-equivariant map from~\(\total{\VB}\) onto an open subset of \(\total{\Triv^\Target} = \total{\Triv}\times_\Base\Target\) that is a homeomorphism with respect to the subspace topology from~\(\total{\Triv^\Target}\)).
  \end{itemize}
  In addition, we assume that the dimensions of the fibres of the \(\Grd\)\nb-vector bundles \(\VB\) and~\(\Triv\) are bounded above by some \(n\in\N\).

  The \emph{trace} of a normally non-singular map is the \(\Grd\)\nb-map
  \[
  f\defeq \proj{\Triv^\Target}\circ\hat{f}\circ\zers{\VB}\colon
  \Source\mono\total{\VB}\opem\total{\Triv^\Target}\epi\Target.
  \]
  Its \emph{degree} is \(\dim \VB-\dim\Triv\) if this locally constant function on~\(\Source\) is constant (otherwise the degree is not defined).  Its \emph{stable normal bundle} is \([\VB]-[\Triv^\Source]\), viewed as an element in the Grothendieck group of the monoid of \(\Coh\)\nb-oriented subtrivial \(\Grd\)\nb-vector bundles on~\(\Source\).

  The normally non-singular \(\Grd\)\nb-map \((\VB,\Triv,\hat{f})\) is called a \emph{normally non-singular embedding} if \(\Triv=0\), so that \(\proj{\Triv^\Target} = \Id_{\Target}\) and \(f=\hat{f}\circ\zers{\VB}\); it is called a \emph{special normally non-singular submersion} if \(\VB=0\), so that \(\zers{\VB}=\Id_\Source\) and \(f=\proj{\Triv^\Target}\circ \hat{f}\).
\end{definition}

The assumption that~\(\Triv\) should be trivial is needed to define the composition of normally non-singular maps~-- this requires extending the \(\Grd\)\nb-vector bundle~\(\Triv^\Target\) to larger spaces, which only works in a canonical way for trivial \(\Grd\)\nb-vector bundles.  As a consequence, a vector bundle projection \(\proj{\VB}\colon \total{\VB}\to\Target\) is the trace of a special normally non-singular submersion if and only if~\(\VB\) is trivial.  If~\(\VB\) is subtrivial, then we may at least lift~\(\proj{\VB}\) to a normally non-singular map (see~\cite{Emerson-Meyer:Normal_maps}).  Since some manipulations with correspondences require~\(\proj{\VB}\) to have such a normally non-singular lifting, we need~\(\VB\) to be subtrivial in Definition~\ref{def:normal_map}.  This assumption is already made in~\cite{Emerson-Meyer:Normal_maps}, but it only becomes relevant here.

An \(\Coh\)\nb-oriented normally non-singular map~\(f\) from~\(\Source\) to~\(\Target\) generates a wrong-way map \(f!\colon \Coh^*_\Base(\Source)\to \Coh^*_\Base(\Target)\), see~\cite{Emerson-Meyer:Normal_maps}.  The notion of equivalence for normally non-singular maps is based on a natural notion of isotopy and on a lifting along trivial \(\Grd\)\nb-vector bundles.  We refer to~\cite{Emerson-Meyer:Normal_maps} for the definition of the composition and exterior product of normally non-singular maps.  The topological wrong-way functoriality \(f\mapsto f!\) is well-defined on equivalence classes and is compatible with composition and exterior products.

Now let \(\Source\) and~\(\Target\) be smooth \(\Grd\)\nb-manifolds (see~\cite{Emerson-Meyer:Normal_maps}).  Then we may also consider smooth normally non-singular maps from~\(\Source\) to~\(\Target\).  For such maps, we require a smooth structure on the \(\Grd\)\nb-vector bundle~\(\VB\) (this is automatic for~\(\Triv\)) and assume that~\(\hat{f}\) is a fibrewise diffeomorphism.  Smooth equivalence for smooth normally non-singular maps is based on smooth isotopies and lifting.  Under suitable technical hypotheses, any smooth map \(\Source\to\Target\) lifts to a smooth normally non-singular map, which is unique up to smooth equivalence.  For instance, this works if~\(\Grd\) is a compact group and~\(\Source\) is compact, or if \(\Grd=G\ltimes\Base\) for a discrete group~\(G\) and a proper \(G\)\nb-CW-complex~\(\Base\) with finite covering dimension and with uniformly bounded size of the isotropy groups (see~\cite{Emerson-Meyer:Normal_maps}).  There are also examples of compact groupoids for which all this fails, that is, there may be smooth maps that do not lift to smooth normally non-singular maps.  These counterexamples oblige us to use normally non-singular maps.

\subsection{The definition of correspondence}
\label{sec:def_correspondence}

\begin{definition}
  \label{def:correspondence}
  A \emph{\textup(\(\Grd\)\nb-equivariant, \(\Coh\)\nb-oriented\textup) correspondence} from~\(\Source\) to~\(\Target\) is a quadruple \((\Midd,\mapl,\mapr,\Kclass)\), where
  \begin{itemize}
  \item \(\Midd\) is a \(\Grd\)\nb-space (\(M\) for middle);
  \item \(\mapl\colon \Midd\to\Source\) is a \(\Grd\)\nb-map (\(\mapl\) for backwards);
  \item \(\mapr\colon \Midd\to\Target\) is an \(\Coh\)\nb-oriented normally non-singular \(\Grd\)\nb-map (\(\mapr\) for forwards);
  \item \(\Kclass\) belongs to \(\Coh^*_\Source(\Midd)\); here we use~\(\mapl\) to view~\(\Midd\) as a space over~\(\Source\).
  \end{itemize}
  The \emph{degree} of a correspondence is the sum of the degrees of \(\mapr\) and~\(\Kclass\) (it need not be defined).

  A correspondence \((\Midd,\mapl,\mapr,\Kclass)\) is called \emph{proper} if \(\mapl\colon \Midd\to\Source\) is proper.  Then any closed subset of~\(\Midd\) --~including~\(\Midd\) itself~-- is \(\Source\)\nb-compact, so that \(\Coh^*_\Source(\Midd) \cong \Coh^*(\Midd)\).
\end{definition}

Our definition deviates from previous ones (see \cites{Baum-Douglas:K-homology, Connes-Skandalis:Longitudinal, Raven:Thesis}) in two aspects: we do not require~\(\mapl\) to be proper, and we let~\(\mapr\) be a normally non-singular map instead of a smooth map.  We have explained in the introduction why these changes are helpful.

\begin{example}
  \label{exa:pullback_correspondence}
  A proper \(\Grd\)\nb-map \(\mapl\colon \Target\to\Source\) yields a correspondence \(\mapl^*\defeq (\Target,\mapl,\Id_\Target,1)\) from~\(\Source\) to~\(\Target\), where \(\Id_\Target\) denotes the identity normally non-singular map on~\(\Target\) and \(1\in\Coh^*(\Target) \cong \Coh^*_\Source(\Target)\) is the unit element.
\end{example}

\begin{example}
  \label{exa:wrong-way_correspondence}
  An \(\Coh\)\nb-oriented normally non-singular \(\Grd\)\nb-map \(\mapr\colon \Source\to\Target\) yields a correspondence \(\mapr!\defeq (\Source,\Id_\Source,\mapr,1)\) from~\(\Source\) to~\(\Target\).
\end{example}

\begin{example}
  \label{exa:multiplier_correspondence}
  Any class~\(\Kclass\) in \(\Coh^*(\Source)\) yields a correspondence \((\Source,\Id_\Source,\Id_\Source,\Kclass)\) from~\(\Source\) to itself.
\end{example}

\begin{definition}
  \label{def:add_correspondences}
  The \emph{sum} of two correspondences is their disjoint union:
  \[
  (\Midd_1,\mapl_1,\mapr_1,\Kclass_1) +
  (\Midd_2,\mapl_2,\mapr_2,\Kclass_2) \defeq
  (\Midd_1\sqcup\Midd_2,\mapl_1\sqcup\mapl_2,\mapr_1\sqcup\mapr_2,
  \Kclass_1\sqcup\Kclass_2).
  \]
\end{definition}

This uses \cite{Emerson-Meyer:Normal_maps}*{Lemma 4.30} and is well-defined, associative, and commutative up to isomorphism.  The empty correspondence with \(\Midd=\emptyset\) acts as zero.

Any correspondence decomposes uniquely as a sum of correspondences of degree~\(j\) for \(j\in\Z\).  To see this, write \(\mapr=(\VB,\Triv,\hat{\mapr})\), and decompose~\(\Midd\) into the disjoint subsets where \(\VB\), \(\mapr^*(\Triv)\), and~\(\Kclass\) have certain degrees.  The dimension assumption ensures that only finitely many non-empty pieces arise.

\subsection{Equivalence of correspondences}
\label{sec:equivalence_correspondence}

Now we define when two correspondences are equivalent.  For this, we introduce several elementary relations, which together generate equivalence.  Besides isomorphism, we need equivalence of the normally non-singular maps, bordism, and Thom modification; the latter replaces the notions of vector bundle modification used in \cites{Baum-Douglas:K-homology, Baum-Block:Bicycles}.  The only reason not to call it by that name is to avoid confusion with the two different notions that already go by it.

It is clear when two correspondences are \emph{isomorphic}.  In the following, we tacitly work with isomorphism classes of correspondences all the time.  \emph{Equivalence of normally non-singular maps} simply means that we consider the correspondences \((\Midd,\mapl,\mapr_0,\Kclass)\) and \((\Midd,\mapl,\mapr_1,\Kclass)\) equivalent if \(\mapr_0\) and~\(\mapr_1\) are equivalent \(\Coh\)\nb-oriented normally non-singular maps.

\begin{definition}
  \label{def:bordism}
  A \emph{bordism of correspondences} from~\(\Source\) to~\(\Target\) consists of
  \begin{itemize}
  \item \(\Bord\), a \(\Grd\)\nb-space;

  \item \(\mapl\), a \(\Grd\)\nb-map from~\(\Bord\) to~\(\Source\);

  \item \(\mapr\defeq (\VB,\Triv,\hat{\mapr})\), an \(\Coh\)\nb-oriented normally non-singular \(\Grd\)\nb-map from~\(\Bord\) to \(\Target\times[0,1]\) -- that is, \(\VB\) is a subtrivial \(\Grd\)\nb-vector bundle over~\(\Bord\), \(\Triv\) is a \(\Grd\)\nb-vector bundle over~\(\Base\), and~\(\hat{\mapr}\) is an open embedding from~\(\total{\VB}\) into \(\total{\Triv^{\Target\times[0,1]}} \cong \total{\Triv^\Target}\times[0,1]\) -- with the additional propery that there are subsets \(\bd_0\Bord,\bd_1\Bord\subseteq\Bord\) such that
    \[
    \hat{\mapr}^{-1}\bigl(\Target\times\{j\}\bigr)
    = \proj{\VB}^{-1}(\bd_j\Bord) \subseteq \total{\VB}
    \qquad\text{for \(j=0,1\);}
    \]

  \item \(\Kclass \in \Coh^*_\Source(\Bord)\).
  \end{itemize}
  Example~\ref{exa:smooth_bordism} explains the relationship to the more traditional notion of bordism.

  A bordism \(\Cor = (\Bord,\mapl,\mapr,\Kclass)\) from~\(\Source\) to~\(\Target\) restricts to correspondences
  \[
  \bd_j\Cor =
  (\bd_j\Bord,\mapl|_{\bd_j\Bord},\mapr|_{\bd_j\Bord}, \Kclass|_{\bd_j\Bord})
  \]
  from~\(\Source\) to~\(\Target\), where \(\mapr|_{\bd_j\Bord}\) denotes the normally non-singular map \((\VB|_{\bd_j\Bord},\Triv,\hat{\mapr}_j)\) from~\(\bd_j\Bord\) to~\(\Target\) with \(\hat{\mapr}(\vb) = (\hat{\mapr}_j(\vb),j)\) for \(\vb\in \proj{\VB}^{-1}(\bd_j\Bord)\), \(j=0,1\).

  We call these correspondences \(\bd_0\Cor\) and~\(\bd_1\Cor\) \emph{bordant} and write
  \[
  \bd_0\Cor \bordant \bd_1\Cor.
  \]
\end{definition}

Finally, we incorporate Thom isomorphisms:

\begin{definition}
  \label{def:Thom_iso_correspondence}
  Let \(\Cor\defeq (\Midd,\mapl,\mapr,\Kclass)\) be a correspondence from~\(\Source\) to~\(\Target\) and let~\(\VB\) be a subtrivial \(\Coh\)\nb-oriented \(\Grd\)\nb-vector bundle over~\(\Midd\).  Let \(\proj{\VB}\colon \total{\VB}\epi\Midd\) be the bundle projection, viewed as an \(\Coh\)\nb-oriented normally non-singular \(\Grd\)\nb-map.  The \emph{Thom modification}~\(\Cor^\VB\) of~\(\Cor\) with respect to~\(\VB\) is the correspondence
  \[
  \bigl(\total{\VB}, \mapl\circ\proj{\VB}, \mapr\circ\proj{\VB}, \Thom_\VB(\Kclass)\bigr);
  \]
  here \(\mapr\circ\proj{\VB}\) denotes the composition of \(\Coh\)\nb-oriented normally non-singular maps and~\(\tau_\VB\) denotes the Thom isomorphism \(\Coh^*_\Source(\Midd) \to \Coh^*_\Source(\total{\VB})\) for~\(\VB\), shifting degrees by \(+\dim (\VB)\) and given by composing pull-back with multiplication by an assumed \emph{Thom class}, or \emph{orientation class} in \(\Coh^{\dim \VB}_\Source(\total{\VB})\) (see \cite{Emerson-Meyer:Normal_maps}*{Definition 5.1}).
  \end{definition}

\begin{definition}
  \label{def:equivalence_correspondence}
  \emph{Equivalence} of correspondences is the equivalence relation on the set of correspondences from~\(\Source\) to~\(\Target\) generated by equivalence of normally non-singular maps, bordism, and Thom modification.  Let \(\Bic^*(\Source,\Target)\) be the set of equivalence classes of correspondences from~\(\Source\) to~\(\Target\).
\end{definition}

Equivalence preserves the degree and the addition of correspondences, so that \(\Coh^*(\Source,\Target)\) becomes a graded monoid.

We will show below that reversing the \(\Coh\)\nb-orientation on~\(\mapr\) provides additive inverses, so that \(\Bic^*(\Source,\Target)\) is a graded Abelian group.

\subsection{Examples of bordisms}
\label{sec:bordism}

We establish that bordism is an equivalence relation and that it contains homotopy for the maps \(\mapl\colon \Midd\to\Source\) and isotopy for the normally non-singular maps \(\mapr\colon \Midd\to\Target\).  We also construct some important examples of bordisms.

\begin{proposition}
  \label{pro:bordism_equivalence_relation}
  The relation~\(\bordant\) is an equivalence relation on correspondences from~\(\Source\) to~\(\Target\).
\end{proposition}

\begin{proof}
  Let \(\Cor=(\Midd,\mapl,(\VB,\Triv,\hat{\mapr}),\Kclass)\) be a correspondence from~\(\Source\) to~\(\Target\).  Define
  \[
  \Bord\defeq \Midd\times[0,1],\quad
  \mapl' \defeq \mapl\circ p,\quad
  \VB'\defeq p^*(\VB),\quad
  \hat{\mapr}'(\vb,t) = \bigl(\hat{\mapr}(\vb),t\bigr),\quad
  \Kclass'\defeq p^*(\Kclass),
  \]
  where \(p\colon \Bord\to\Midd\) is the coordinate projection.  Then \((W,\mapl',(\VB',\Triv,\hat{\mapr}'),\Kclass')\) is a bordism between~\(\Cor\) and itself, so that~\(\bordant\) is reflexive.

  If \(\Cor = (\Bord,\mapl,(\VB,\Triv,\hat{\mapr}),\Kclass)\) is a bordism from~\(\Source\) to~\(\Target\), so is \((\Bord,\mapl,(\VB,\Triv,\sigma\circ \hat{\mapr}),\Kclass)\), where \(\sigma\colon \total{\Triv^\Target}\times[0,1]\to\total{\Triv^\Target}\times[0,1]\) maps \((\triv,t)\) to \((\triv,1-t)\).  This exchanges the roles of \(\bd_0\Cor\) and \(\bd_1\Cor\), proving that~\(\bordant\) is symmetric.

  Let \(\Cor_1 = (\Bord_1,\mapl_1,(\VB_1,\Triv_1,\hat{\mapr}_1),\Kclass_1)\) and \(\Cor_2 = (\Bord_2,\mapl_2,(\VB_2,\Triv_2,\hat{\mapr}_2),\Kclass_2)\) be bordisms such that the correspondences \(\bd_1\Cor_1\) and \(\bd_0\Cor_2\) are isomorphic.  Hence \(\Triv_1\cong\Triv_2\) -- we may even assume \(\Triv_1=\Triv_2\) -- and there is a homeomorphism \(\bd_1\Bord_1 \cong \bd_0\Bord_2\) compatible with the other structure.  It allows us to glue together \(\Bord_1\) and~\(\Bord_2\) to a \(\Grd\)\nb-space \(\Bord_{12} \defeq \Bord_1 \cup_{\bd_1\Bord_1 \cong \bd_0\Bord_2} \Bord_2\) and \(\mapl_1\) and~\(\mapl_2\) to a \(\Grd\)\nb-map \(\mapl_{12}\colon \Bord_{12}\to\Source\).  The \(\Grd\)\nb-vector bundles \(\VB_1\) and~\(\VB_2\) combine to a \(\Grd\)\nb-vector bundle~\(\VB_{12}\) on~\(\Bord_{12}\), which inherits an \(\Coh\)\nb-orientation by \cite{Emerson-Meyer:Normal_maps}*{Lemma 5.6}.  The classes \(\Kclass_1\) and~\(\Kclass_2\) combine to a class \(\Kclass_{12}\in\Coh^*_\Source(\Bord_{12})\) by the Mayer--Vietoris sequence for~\(\Coh^*_\Source\).  Rescale \(\hat{\mapr}_1\) and~\(\hat{\mapr}_2\) to open embeddings from \(\total{\VB_1}\) and~\(\total{\VB_2}\) to \(\total{\Triv^\Target}\times[0,\nicefrac12]\) and \(\total{\Triv^\Target}\times[\nicefrac12,1]\) that map \(\proj{\VB_1}^{-1}(\bd_0\Bord_1)\) to \(\total{\Triv^\Target}\times\{0\}\), \(\proj{\VB_1}^{-1}(\bd_1\Bord_1)\) and \(\proj{\VB_2}^{-1}(\bd_0\Bord_2)\) to \(\total{\Triv^\Target}\times\{\nicefrac12\}\), and \(\proj{\VB_2}^{-1}(\bd_1\Bord_2)\) to \(\total{\Triv^\Target}\times\{1\}\).  These combine to an open embedding~\(\hat{\mapr}_{12}\) from~\(\total{\VB_{12}}\) into \(\total{\Triv^\Target}\times[0,1]\).  This yields a bordism \((\Bord_{12},\mapl_{12},(\VB_{12},\Triv_{12},\hat{\mapr}_{12}), \Kclass_{12})\) from~\(\bd_0\Cor_1\) to~\(\bd_1\Cor_2\).  Thus~\(\bordant\) is transitive.
\end{proof}

\begin{lemma}
  \label{lem:bordism_homotopy}
  Let \((\Midd,\mapl_0,\mapr_0,\Kclass)\) be a correspondence from~\(\Source\) to~\(\Target\).  Let~\(\mapl_0\) be homotopic to~\(\mapl_1\) and let~\(\hat{\mapr}_0\) be isotopic to~\(\hat{\mapr}_1\).  Then the correspondences \((\Midd,\mapl_0,\mapr_0,\Kclass)\) and \((\Midd,\mapl_1,\mapr_1,\Kclass)\) are bordant.
\end{lemma}

\begin{proof}
  The bordism is constructed as in the proof that bordism is reflexive; but this time, \(\mapl'\) is replaced by a homotopy between~\(\mapl_0\) and~\(\mapl_1\), and~\(\hat{\mapr}'\) by an isotopy between \(\hat{\mapr}_0\) and~\(\hat{\mapr}_1\).
\end{proof}

\begin{example}
  \label{exa:smooth_bordism}
  Let \(\Source\) and~\(\Target\) be smooth manifolds and let~\(\Bord\) be a smooth manifold with boundary~\(\bd\Bord\), decomposed into two disjoint subsets: \(\bd\Bord=\bd_0\Bord\sqcup\bd_1\Bord\).  Let \(\Kclass\in\Coh^*_\Source(\Bord)\), let \(\mapl\colon \Bord\to\Source\) be a smooth map, and let \(\mapr\colon \Bord\to\Target\) be a smooth map that is \(\Coh\)\nb-oriented in the sense that \(\mapr^*(\Tvert\Target) \oplus \Normal_\Bord\) is \(\Coh\)\nb-oriented or, equivalently, \([\Tvert\Source]-f^*[\Tvert\Target]\) is stably \(\Coh\)\nb-oriented.  We want to construct a bordism from this data.

  We define the stable normal bundle~\(\Normal_\Bord\) as the restriction of~\(\Normal_{D\Bord}\) to~\(\Bord\), where \(D\Bord \defeq \Bord\cup_{\bd\Bord}\Bord\) is the double of~\(\Bord\) -- a smooth manifold.  Recall that \(\Normal_{D\Bord}\) is the normal bundle of a smooth embedding \(h\colon D\Bord\to\R^n\) for some \(n\in\N\).  We lift~\(\mapr\) to a normally non-singular map \(\NM=(\VB,\R^n,\hat{\mapr})\) from~\(\Bord\) to \(\Target\times[0,1]\) as follows.  Let \(k\colon \Bord\to[0,1]\) be a smooth map with \(\bd_j\Bord=k^{-1}(j)\) for \(j=0,1\) and with non-vanishing first derivative on~\(\bd\Bord\).  Then \((\mapr,h|_\Bord,k)\colon \Bord\to \Target\times\R^n\times[0,1]\) identifies~\(\Bord\) with a neat submanifold of \(\Target\times\R^n\times[0,1]\) (see \cite{Hirsch:Diff_Top}*{page 30}).  The Tubular Neighbourhood Theorem for smooth manifolds with boundary shows that \((\mapr,h|_\Bord,k)\) extends to a diffeomorphism~\(\hat{\mapr}\) from its normal bundle \(\VB \cong \mapr^*(\Tvert\Target)\oplus \Normal_\Bord\oplus\R\) onto an open subset of \(\Target\times\R^n\times[0,1]\).  We get an \(\Coh\)\nb-oriented normally non-singular map \(\NM\defeq (\VB,\hat{\mapr},\R^n)\) from~\(\Bord\) to \(\Target\times[0,1]\).

  Putting everything together, we get a bordism of correspondences \((\Bord,\mapl,\NM,\Kclass)\) with \(\bd_0\Bord\) and~\(\bd_1\Bord\) as specified.  Furthermore, the trace of~\(\NM\) lifts~\(\mapr\) to a map \(\Bord\to \Target\times[0,1]\), so that \(\NM|_{\bd_j\Bord}\) is a normally non-singular map with trace~\(\mapr|_{\bd_j\Bord}\colon \bd_j\Bord\to\Target\).
\end{example}

\begin{example}
  \label{exa:bordism_extend_from_open}
  Let \((\Midd,\mapl,\mapr,\Kclass)\) be a correspondence from~\(\Source\) to~\(\Target\).  Let \(\Midd'\subseteq\Midd\) be an open \(\Grd\)\nb-invariant subset and assume that there is \(\Kclass'\in\Coh^*_\Source(\Midd')\) that is mapped to~\(\Kclass\) by the canonical map \(\Coh^*_\Source(\Midd')\to\Coh^*_\Source(\Midd)\).  We claim that the correspondences \((\Midd,\mapl,\mapr,\Kclass)\) and \((\Midd',\mapl|_{\Midd'},\mapr|_{\Midd'},\Kclass')\) are bordant.  Here \(\mapr|_{\Midd'}\) denotes the composition of~\(\mapr\) with the open embedding \(\Midd'\opem\Midd\), viewed as a normally non-singular map; if \(\mapr=(\VB,\Triv,\hat{\mapr})\), then \(\mapr|_{\Midd'} = (\VB|_{\Midd'},\Triv,\hat{\mapr}|_{\Midd'})\).

  The underlying space of the bordism is the \(\Grd\)\nb-invariant open subset
  \[
  \Bord \defeq \Midd'\times\{0\}\cup \Midd\times(0,1] \subseteq \Midd\times[0,1]
  \]
  with the subspace topology, induced \(\Grd\)\nb-action, and the obvious maps to \(\Source\) and~\(\Target\) (see the proof of Proposition~\ref{pro:bordism_equivalence_relation}).  We may pull back~\(\Kclass'\) to a class in \(\Coh^*_\Source(\Midd'\times[0,1])\), which then extends to a class in \(\Coh^*_\Source(\Bord)\) whose restrictions to \(\Midd'\times\{0\}\) and \(\Midd\times\{1\}\) are \(\Kclass'\) and~\(\Kclass\), respectively.  This yields the required bordism between \((\Midd,\mapl,\mapr,\Kclass)\) and \((\Midd',\mapl|_{\Midd'},\mapr|_{\Midd'},\Kclass')\).
\end{example}

It is unclear from our definition of bordism which subsets \(\bd_0\Bord\) and~\(\bd_1\Bord\) of~\(\Bord\) are possible.  The following definition provides a criterion for this:

\begin{definition}
  \label{def:boundary}
  Let~\(\Bord\) be a \(\Grd\)\nb-space.  A closed \(\Grd\)\nb-invariant subset \(\bd\Bord\) is called a \emph{boundary} of~\(\Bord\) if the embedding \(\bd\Bord\times\{0\} \cong \bd\Bord\to\Bord\) extends to a \(\Grd\)\nb-equivariant open embedding \(c\colon \bd\Bord\times[0,1)\opem\Bord\); the map~\(c\) is called a \emph{collar} for~\(\bd\Bord\).
\end{definition}

If \(\bd\Bord\subseteq\Bord\) is a boundary, then we let \(\Bord^\circ\defeq \Bord\setminus \bd\Bord\) be the \emph{interior} of~\(\Bord\).

We identify \(\bd\Bord\times[0,1)\) with a subset of~\(\Bord\) using the collar.  The following lemma uses the auxiliary orientation-preserving diffeomorphism:
\[
\varphi\colon \R \xrightarrow{\cong} (0,1),
\qquad
t \mapsto \frac12+\frac{t}{2\sqrt{1+t^2}}.
\]
Notice that \(\varphi(-t) = 1 -\varphi(t)\).

\begin{lemma}
  \label{lem:bordism_from_boundary}
  Let \(\bd_0\Bord\sqcup\bd_1\Bord\subseteq \Bord\) be a boundary.  Then there is an open embedding \(h\colon \Bord\times\R \opem \Bord^\circ\times[0,1]\) with the following properties:
  \begin{itemize}
  \item \(h(\bord,t) = \bigl(\bord,\varphi(t)\bigr)\) for \(\bord\notin \bd\Bord\times[0,\nicefrac12)\);

  \item \(h(w,t)\in W^\circ\times(0,1)\) for \(w\in W\setminus\bd\Bord\);

  \item \(h\bigl((\bord,0),t\bigr) = \Bigl(\bigl(\bord,\varphi(-t)/2\bigr),0\Bigr)\) for \(\bord\in\bd_0\Bord\);

  \item \(h\bigl((\bord,0),t\bigr) = \Bigl(\bigl(\bord,\varphi(t)/2\bigr),1\Bigr)\) for \(\bord\in\bd_1\Bord\).
  \end{itemize}
\end{lemma}

\begin{proof}
  Let \(A\defeq \Bord \setminus \bd\Bord\times [0,\nicefrac12]\).  We put \(h(\bord,t) \defeq \bigl(\bord, \varphi(t)\bigr)\) if \(\bord\in A\) to fulfil the first condition; this maps \(A\times\R\) homeomorphically onto \(A\times (0,1)\).  On \(\bd\Bord\times[0,\nicefrac12]\times\R\), we connect the prescribed values on \(\bd\Bord\times\{0, \nicefrac12\}\times\R\) by an affine homotopy; that is, if \(\bord\in\bd\Bord\), \(s\in[0,\nicefrac12)\), and \(t\in\R\), then
  \[
  h\bigl((\bord,s),t\bigr) \defeq
  \begin{cases}
    \Bigl(\bigl(\bord,
    s - (s-\nicefrac12)\varphi(-t)\bigr),
    2s\varphi(t)\Bigr)&\text{if \(\bord\in\bd_0\Bord\),}\\
    \Bigl(\bigl(\bord,
    s - (s-\nicefrac12)\varphi(t)\bigr),
    1-2s\varphi(-t)\Bigr)&\text{if \(\bord\in\bd_1\Bord\).}\\
  \end{cases}
  \]

  A routine computation shows that~\(h\) maps \(\bd\Bord\times(0,\nicefrac12]\times\R\) homeomorphically onto a relatively open subset of itself.  Hence~\(h\) is an open embedding on \(\Bord\times\R\).
\end{proof}

\begin{example}
  \label{exa:disk_bundle_bordism}
  Let \((\Midd,\mapl,\mapr,\Kclass)\) be a correspondence from~\(\Source\) to~\(\Target\) and let~\(\VB\) be a subtrivial \(\Coh\)\nb-oriented \(\Grd\)\nb-vector bundle over~\(\Midd\) equipped with some \(\Grd\)\nb-invariant inner product.  Let \(\Sphere\VB\subseteq \Disk\VB\subseteq\total{\VB}\) be the unit sphere and unit disk bundles and let \(\proj{\Disk}\colon \Disk\VB\to\Midd\) and \(\proj{\Sphere}\colon \Sphere\VB\to\Midd\) be the canonical projections.

  The projection \(\proj{\VB}\colon \total{\VB}\epi\Midd\) is an \(\Coh\)\nb-oriented normally non-singular map by \cite{Emerson-Meyer:Normal_maps}*{Example 4.25}.  The embedding \(\Sphere\VB\to\total{\VB}\) is a normally non-singular embedding with constant normal bundle~\(\R\) for a suitable tubular neighbourhood, say,
  \begin{equation}
    \label{eq:sphere_tube}
    \Sphere\VB\times\R\opem \total{\VB},\qquad
    (\vb,t)\mapsto \vb\cdot \bigl(2-2\varphi(t)\bigr),
  \end{equation}
  with the auxiliary function~\(\varphi\) above.  Hence~\(\proj{\Sphere}\) is an \(\Coh\)\nb-oriented normally non-singular map.  We get a correspondence \(\bigl(\Sphere\VB,\mapl\circ \proj{\Sphere},\mapr\circ \proj{\Sphere},\proj{\Sphere}^*(\Kclass)\bigr)\) from~\(\Source\) to~\(\Target\).  We claim that this correspondence is bordant to the empty correspondence.

  We want to construct a bordism \((\Bord,\mapl',\mapr',\Kclass')\) with
  \[
  \Bord=\Disk\VB,\qquad
  \mapl'\defeq \mapl\circ \proj{\Disk},\qquad
  \Kclass'\defeq \proj{\Disk}^*(\mapl),\qquad
  \bd_0\Bord=\emptyset,\qquad
  \bd_1\Bord = \Sphere\VB.
  \]
  The \(\Coh\)\nb-oriented normally non-singular map \(\mapr\circ \proj{\VB}\colon \total{\VB}\to\Target\) pulls back to an \(\Coh\)\nb-oriented normally non-singular map \((\mapr \proj{\VB})\times[0,1]\colon \total{\VB}\times[0,1]\to\Target\times[0,1]\).  We let~\(\mapr'\) be the composition of \((\mapr \proj{\VB})\times[0,1]\) with the \(\Coh\)\nb-oriented normally non-singular embedding \((\Disk\VB\times\R,h)\) from~\(\Disk\VB\) to \(\total{\VB}\times[0,1]\), where \(h\colon \Disk\VB\times\R\to (\Disk\VB)^\circ\times[0,1] \subseteq \total{\VB}\times[0,1]\) is the open embedding constructed in Lemma~\ref{lem:bordism_from_boundary}.  Here we use the collar \(\Sphere\VB\times[0,1)\opem\Disk\VB\), \((\vb,t)\mapsto \vb\cdot(1-t)\).  The open embedding \(h|_{\Sphere\VB}\colon \Sphere\VB\times\R\opem\total{\VB}\) is isotopic to the tubular neighbourhood for~\(\Sphere\VB\) in~\eqref{eq:sphere_tube}.  Hence the boundary of \((\Bord,\mapl',\mapr',\Kclass')\) is equivalent to \(\bigl(\Sphere\VB,\mapl\circ \proj{\Sphere},\mapr\circ \proj{\Sphere},\proj{\Sphere}^*(\Kclass)\bigr)\).
\end{example}

\begin{example}
  \label{exa:inverse_correspondence}
  Let \(\Cor = (\Midd,\mapl,\mapr,\Kclass)\) be a correspondence from~\(\Source\) to~\(\Target\).  Let~\(-\mapr\) denote~\(\mapr\) with the opposite \(\Coh\)\nb-orientation and let \(-\Cor \defeq (\Midd,\mapl, -\mapr,\Kclass)\).  Up to bordism, this is inverse to~\(\Cor\), that is, the disjoint union \(\Cor\sqcup -\Cor\) is bordant to the empty correspondence.  As a consequence, bordism classes of \(\Grd\)\nb-equivariant correspondences from~\(\Source\) to~\(\Target\) form an Abelian group.

  The bordism \(\Cor\sqcup -\Cor\bordant \emptyset\) is, in fact, a special case of Example~\ref{exa:disk_bundle_bordism} where \(\VB\defeq\Midd\times\R\) is the constant vector bundle of rank~\(1\); hence the disk bundle~\(\Disk\VB\) is simply \(\Midd\times[0,1]\) and the unit sphere bundle is \(\Midd\sqcup\Midd\).  The sign comes from the orientation-reversal on one boundary component in Lemma~\ref{lem:bordism_from_boundary}.
\end{example}

The last two examples allow us to relate the Thom modification in Definition~\ref{def:Thom_iso_correspondence} to the vector bundle modifications used in \cite{Raven:Thesis} and \cites{Baum-Douglas:K-homology,Baum-Block:Bicycles}.

Let \(\Cor\defeq (\Midd,\mapl,\mapr,\Kclass)\) be a correspondence from~\(\Source\) to~\(\Target\) and let~\(\VB\) be a subtrivial \(\Coh\)\nb-oriented \(\Grd\)\nb-vector bundle over~\(\Midd\).  Since the bundle projection \(\proj{\VB}\colon \total{\VB}\epi\Midd\) is not proper, the Thom modification makes no sense in the setting of \cites{Baum-Douglas:K-homology, Baum-Block:Bicycles, Raven:Thesis}.  Let~\(\bar\VB\) be the unit sphere bundle in \(\VB\oplus\R\).  This contains~\(\total{\VB}\) as an open subset, whose complement is homeomorphic to~\(\Midd\) via the \(\infty\)\nb-section.  Excision for~\(\Coh\) yields a canonical map
\[
\Coh^*_\Source(\total{\VB}) \cong
\Coh^*_\Source(\bar\VB,\Midd) \to
\Coh^*_\Source(\bar\VB).
\]
Let \(\bar\Thom_\VB\colon \Coh^*_\Source(\Midd)\to\Coh^*_\Source(\bar\VB)\) be its composition with the Thom isomorphism.  The projection \(\proj{\VB}\colon \total{\VB}\epi\Midd\) extends to an \(\Coh\)\nb-oriented normally non-singular map \(\bproj{\VB}\colon \bar\VB\to\Midd\).  We get a correspondence \(\bigl(\bar\VB, \mapl\circ\bproj{\VB}, \mapr\circ\bproj{\VB}, \bar\Thom_\VB(\Kclass)\bigr)\).  This is precisely the vector bundle modification used by Jeff Raven in~\cite{Raven:Thesis}.  Example~\ref{exa:bordism_extend_from_open} shows that
\[
\bigl(\VB, \mapl\circ\proj{\VB}, \mapr\circ\proj{\VB}, \Thom_\VB(\Kclass)\bigr)
\bordant
\bigl(\bar\VB, \mapl\circ\bproj{\VB}, \mapr\circ\bproj{\VB}, \bar\Thom_\VB(\Kclass)\bigr).
\]
Thus the Thom modification is bordant to Raven's vector bundle modification of~\(\Cor\).

The notion of vector bundle modification in \cites{Baum-Douglas:K-homology, Baum-Block:Bicycles} is slightly different from Raven's.  The clutching construction in~\cite{Baum-Douglas:K-homology} does not involve the full Thom class, it only uses its non-trivial half.  Recall that the Thom class \(\Thom_\VB\in\RK^*_{\Grd,\Midd}(\bar\VB,\Midd)\) in \(\K\)\nb-theory restricts to the Bott generator in \(\K^n(\Sphere^n,\pt)\) in each fibre.  Since the dimension vanishes on this relative \(\K\)\nb-group, the Thom class is a difference of two vector bundles.  One is the clutching construction of~\cite{Baum-Douglas:K-homology}, the other is pulled back from~\(\Midd\).  Leaving out this second half yields a bordant correspondence because of Example~\ref{exa:disk_bundle_bordism}, which yields a bordism
\[
\bigl(\bar\VB, \mapl\circ\bproj{\VB}, \mapr\circ\bproj{\VB}, \bproj{\VB}^*(\delta)\bigr) \bordant \emptyset
\]
for any \(\delta\in\Coh^*_\Source(\Midd)\).

This is why the two notions of vector bundle modification used by Baum and Raven are almost equivalent.  The only difference is that -- unlike Baum's -- Raven's vector bundle modification contains the direct sum--disjoint union relation when combined with bordism (see \cite{Raven:Thesis}*{Proposition 4.3.2}).

\begin{lemma}
  \label{lem:direct_sum_disjoint_union}
  Let \(\Cor_1 = (\Midd,\mapl,\mapr,\Kclass_1)\) and \(\Cor_2 = (\Midd,\mapl,\mapr,\Kclass_2)\) be two correspondences from~\(\Source\) to~\(\Target\) with the same data \((\Midd,\mapl,\mapr)\) and let \(\Cor_+ \defeq (\Midd,\mapl,\mapr,\Kclass_1 +\Kclass_2)\).  The correspondences \(\Cor_1 \sqcup \Cor_2\) and \(\Cor_+\) are equivalent.
\end{lemma}

\begin{proof}
  We are going to construct a bordism between the Thom modifications of \(\Cor_1 \sqcup \Cor_2\) and \(\Cor_+\) along the constant \(1\)\nb-dimensional \(\Grd\)\nb-vector bundle~\(\R\).

  Let \(\Bord\defeq [0,1]\times\R\setminus\{(0,0)\}\) and \(\bd_j\Bord \defeq \Bord\cap \{j\}\times\R\).  Thus \(\bd_0\Bord \cong \R\sqcup\R\) and \(\bd_1\Bord \cong\R\).  The bordism we seek is of the form \((\Bord\times\Midd,\mapl\circ\pi_2,\mapr',\Kclass)\) for a certain \(\Kclass\in\Coh^*_{\Source}(\Bord\times\Midd)\).  Here \(\pi_2\colon \Bord\times\Midd\to\Midd\) is the canonical projection, and~\(\mapr'\) is the exterior product the open embedding \(\Bord \opem [0,1]\times\R\) with~\(\mapr\).  Excision shows that restriction to~\(\bd_0\Bord\) induces an isomorphism
  \[
  \Coh^*_\Source(\Bord\times\Midd)
  \cong \Coh^*_\Source\bigl((\R\sqcup\R)\times\Midd\bigr)
  \cong \Coh^{*+1}_\Source(\Midd)\oplus \Coh^{*+1}_\Source(\Midd).
  \]
  Hence there is a unique \(\Kclass\in\Coh^*_\Source(\Bord\times\Midd)\) whose restriction to~\(\bd_0\Bord\) is \(\Kclass_1\sqcup\Kclass_2\) and whose restriction to~\(\bd_1\Bord\) is \(\Kclass_1+\Kclass_2\).  This provides the desired bordism between the Thom modifications of \(\Cor_1 \sqcup \Cor_2\) and \(\Cor_+\) along~\(\R\).
\end{proof}

\subsection{Special correspondences}
\label{sec:special_correspondences}

We use Thom modifications to bring correspondences into a standard form.  This greatly simplifies the definition of \(\Bic^*(\Source,\Target)\) and is needed for the composition product.

\begin{definition}
  \label{def:special_correspondence}
  A correspondence \((\Midd,\mapl,\mapr,\Kclass)\) or a bordism \((\Bord,\mapl,\mapr,\Kclass)\) is called \emph{special} if~\(\mapr\) is a special normally non-singular submersion.
\end{definition}

\begin{example}
  \label{ex:special_map}
  The correspondence~\(\mapl^*\) for a proper \(\Grd\)\nb-map \(\mapl\colon \Midd \to \Source\) described in Example~\ref{exa:pullback_correspondence} is special.
\end{example}

Recall that a special normally non-singular submersion from~\(\Source\) to~\(\Target\) is a normally non-singular map of the form \((\Source,\hat{\mapr},\Triv)\), where~\(\hat{\mapr}\) identifies~\(\Source\) with an open subset of~\(\total{\Triv^\Target}\).  In a special correspondence, we may replace~\(\Midd\) by this subset of~\(\total{\Triv^\Target}\), so that~\(\hat{\mapr}\) becomes the identity map.  Hence a special correspondence from~\(\Source\) to~\(\Target\) is equivalent to a quadruple \((\Triv,\Midd,\mapl,\Kclass)\), where
\begin{itemize}
\item \(\Triv\) is an \(\Coh\)\nb-oriented \(\Grd\)\nb-vector bundle over~\(\Base\);

\item \(\Midd\) is an open subset of~\(\total{\Triv^\Target}\);

\item \(\mapl\) is a \(\Grd\)\nb-map from~\(\Midd\) to~\(\Source\);

\item \(\Kclass\in\Coh^*_\Source(\Midd)\);
\end{itemize}
a special bordism from~\(\Source\) to~\(\Target\) is equivalent to a quadruple \((\Triv,\Bord,\mapl,\Kclass)\), where
\begin{itemize}
\item \(\Triv\) is an \(\Coh\)\nb-oriented \(\Grd\)\nb-vector bundle over~\(\Base\);

\item \(\Bord\) is an open subset of~\(\total{\Triv^\Target}\times[0,1]\);

\item \(\mapl\) is a \(\Grd\)\nb-map from~\(\Bord\) to~\(\Source\);

\item \(\Kclass\in\Coh^*_\Source(\Bord)\).
\end{itemize}
Observe that \(\bd_t\Bord \defeq \Bord\cap \bigl(\total{\Triv^\Target}\times\{t\}\bigr)\) for \(t=0,1\) -- viewed as open subsets of~\(\total{\Triv^\Target}\) -- automatically have the properties required in Definition~\ref{def:bordism}.  This is why a special bordism from~\(\Source\) to~\(\Target\) is nothing but a special correspondence from~\(\Source\) to \(\Target\times[0,1]\) (this is not true for general bordisms).  The boundaries of a special bordism are the special correspondences \((\Triv,\bd_t\Bord,\mapl|_{\bd_t\Bord},\Kclass)\) for \(t=0,1\).

Thom modifications of special correspondences need not be special any more, unless we modify by a trivial \(\Grd\)\nb-vector bundle (compare \cite{Emerson-Meyer:Normal_maps}*{Example 4.25}).

\begin{definition}
  \label{def:trivial_Thom}
  Let \(\Cor = (\Triv,\Midd,\mapl,\Kclass)\) be a special correspondence from~\(\Source\) to~\(\Target\) and let~\(\VB\) be an \(\Coh\)\nb-oriented \(\Grd\)\nb-vector bundle over~\(\Base\).  The \emph{Thom modification} of~\(\Cor\) by~\(\VB\) is the special correspondence
  \[
  \Cor^\VB \defeq \bigl(\Triv\oplus\VB, \Midd\times_\Base\total{\VB}, \mapl\circ\proj{\VB^\Midd}, \Thom_\VB(\Kclass)\bigr)
  \]
  from~\(\Source\) to~\(\Target\), where \(\Thom_\VB\colon \Coh^*_\Source(\Midd)\to\Coh^*_\Source(\total{\VB^\Midd}) = \Coh^*_\Source(\Midd\times_\Base \total{\VB})\) is the Thom isomorphism for the induced \(\Coh\)\nb-orientation on~\(\VB^\Midd\).
\end{definition}

\begin{theorem}
  \label{the:special_bordism}
  Any correspondence \((\Midd,\mapl,\mapr,\Kclass)\) is equivalent to a special correspondence.  Two special correspondences are equivalent if and only if they have specially bordant Thom modifications by \(\Grd\)\nb-vector bundles over~\(\Base\).
\end{theorem}

\begin{proof}
  Let \(\Cor= (\Midd,\mapl,\mapr,\Kclass)\) be a correspondence from~\(\Source\) to~\(\Target\) and let \(\mapr= (\VB,\Triv,\hat{\mapr})\), where~\(\VB\) is a subtrivial \(\Coh\)\nb-oriented \(\Grd\)\nb-vector bundle over~\(\Midd\), \(\Triv\) is an \(\Coh\)\nb-oriented \(\Grd\)\nb-vector bundle over~\(\Base\), and~\(\hat{\mapr}\) is an open embedding from~\(\total{\VB}\) into~\(\total{\Triv^\Target}\).  The Thom modification of~\(\Cor\) along~\(\VB\) is a correspondence that involves the composite normally non-singular map \(\total{\VB}\epi\Midd\xrightarrow{\mapr}\Target\), which is equivalent to the special normally non-singular submersion \((\hat{\mapr},\Triv)\) (see \cite{Emerson-Meyer:Normal_maps}*{Example 4.24}).  This yields a special correspondence equivalent to~\(\Cor\).

  We may do the same to a bordism \(\Cor= (\Bord,\mapl,\mapr,\Kclass)\); write \(\mapr=(\VB,\Triv,\hat{\mapr})\), then the Thom modification of~\(\Bord\) along~\(\VB\) is a special bordism, whose boundaries are the Thom modifications of the boundaries \(\bd_0\Cor\) and~\(\bd_1\Cor\) of~\(\Cor\) along the restrictions of~\(\VB\) to \(\bd_0\Bord\) and~\(\bd_1\Bord\), respectively.

  For special correspondences \(\Cor_1\) and~\(\Cor_2\), we write \(\Cor_1\sbordant \Cor_2\) if there is a special bordism between \(\Cor_1\) and~\(\Cor_2\), and \(\Cor_1\sequiv \Cor_2\) if there are \(\Grd\)\nb-vector bundles \(\VB_1\) and~\(\VB_2\) over~\(\Base\) with \(\Cor^{\VB_1} \sbordant \Cor^{\VB_2}\).  An argument as in the proof of Proposition~\ref{pro:bordism_equivalence_relation} shows that~\(\sbordant\) is an equivalence relation.  We claim that~\(\sequiv\) is an equivalence relation as well.  This follows as in the proof of \cite{Emerson-Meyer:Normal_maps}*{Lemma 4.16} using the following observation: if \(\VB_1\) and~\(\VB_2\) are two \(\Grd\)\nb-vector bundles over~\(\Base\), then the Thom modification along~\(\VB_1\) followed by the Thom modification along~\(\VB_2\) yields the Thom modification along \(\VB_1\oplus\VB_2\).

  It is clear that two special correspondences are equivalent if \(\Cor_1\sequiv\Cor_2\) because Thom modification by \(\Grd\)\nb-vector bundles over~\(\Base\) and special bordism are contained in the relations that generate the equivalence of correspondences.  To show that the two relations are equal, we must check the following.  Let \(\Cor_1\) and~\(\Cor_2\) be correspondences and let \(\Cor_1'\) and~\(\Cor_2'\) be the associated special correspondences as above.  If \(\Cor_1\) and~\(\Cor_2\) are related by an equivalence of normally non-singular maps, a bordism, or a Thom modification, then \(\Cor_1'\sequiv \Cor_2'\).  We may further split up equivalence of normally non-singular maps into isotopy and lifting of normally non-singular maps and only have to consider these two special cases.  Let \(\Cor_1 = (\Midd,\mapl,\mapr,\Kclass)\) with \(\mapr=(\VB,\Triv,\hat{\mapr})\).

  We have already observed above that bordism is contained in~\(\sequiv\).  This also covers isotopy of normally non-singular maps, which is a special case of bordism by Lemma~\ref{lem:bordism_homotopy}.  Now suppose that \(\Cor_2= (\Midd,\mapl,\mapr^{\Triv_2},\Kclass)\), where \(\mapr^{\Triv_2} = (\VB\oplus \Triv_2^\Midd, \hat{\mapr}^{\Triv_2},\Triv\oplus\Triv_2)\) is the lifting of~\(\mapr\) along a \(\Grd\)\nb-vector bundle~\(\Triv_2\) over~\(\Base\).  Then \(\Cor_2'\) is the Thom modification of~\(\Cor_1'\) along~\(\Triv_2\), so that \(\Cor_1' \sequiv \Cor_2'\).

  Finally, let~\(\Cor_2\) be the Thom modification of~\(\Cor_1\) along some subtrivial \(\Grd\)\nb-vector bundle~\(\VB_2\) over~\(\Midd\), that is, \(\Cor_2= \bigl(\VB_2,\mapl\circ\proj{\VB_2}, \mapr\circ\proj{\VB_2}, \Thom_{\VB_2}(\Kclass)\bigr)\).  Let \(\VB_2^\bot\) and~\(\Triv_2\) be \(\Grd\)\nb-vector bundles over \(\Midd\) and~\(\Base\) with \(\VB_2\oplus\VB_2^\bot \cong \Triv_2^\Midd\), let \(\iota\colon \total{\VB_2}\oplus\VB_2^\bot\to\Triv_2^\Midd\) be the isomorphism.  Then
  \[
  \mapr\circ\proj{\VB_2} =
  \bigl(\proj{\VB_2}^*(\VB_2^\bot\oplus \VB),
  (\Id_{\Triv_2}\times_\Base\hat{\mapr})\circ (\iota\times_\Midd\Id_{\total{\VB}}),
  \Triv_2\oplus\Triv\bigr);
  \]
  here we use that \(\iota\times_\Midd\Id_{\total{\VB}}\) identifies the total space of \(\proj{\VB_2}^*(\VB_2^\bot \oplus \VB)\) with
  \[
  \total{\proj{\VB_2}^*(\VB_2^\bot\oplus\VB)}
  \cong \total{\VB_2}\times_\Midd \total{\VB_2^\bot}\times_\Midd \total{\VB}
  \cong \total{\Triv_2^\Midd} \times_\Midd\total{\VB} \cong \Triv_2 \times_\Base \total{\VB}.
  \]
  Now it is routine to check that~\(\Cor_2'\) is isomorphic to the Thom modification of~\(\Cor_1'\) along~\(\Triv_2\).
\end{proof}

\begin{theorem}
  \label{the:geometric_K}
  There is a natural isomorphism
  \[
  \Bic^*(\Base,\Target) \cong \Coh_\Base^*(\Target)
  \qquad \text{for all \(\Grd\)\nb-spaces~\(\Target\).}
  \]
\end{theorem}

\begin{proof}
  Theorem~\ref{the:special_bordism} shows that \(\Bic^*(\Base,\Target)\) is the set of \(\sequiv\)\nb-equivalence classes of special correspondences from~\(\Base\) to~\(\Target\).  Let \((\Triv,\Midd,\mapl,\Kclass)\) be a special correspondence from~\(\Base\) to~\(\Target\) as above.  The map~\(\mapl\) must be the anchor map of~\(\Midd\) by \(\Grd\)\nb-equivariance and therefore extends to \(\anchor\colon \total{\Triv^\Target}\to\Base\).  Thus Example~\ref{exa:bordism_extend_from_open} provides a special bordism
  \[
  (\Triv,\Midd,\mapl,\Kclass) \sbordant (\Triv,\total{\Triv^\Target}, \anchor, \bar\Kclass),
  \]
  where \(\bar\Kclass\in\Coh^*_\Base(\total{\Triv^\Target})\) is the image of~\(\Kclass\) under the canonical map \(\Coh^*_\Base(\Midd) \to \Coh^*_\Base(\total{\Triv^\Target})\).  Hence we may restrict attention to special correspondences with \(\Midd=\total{\Triv^\Target}\).  The same argument for special bordisms between such correspondences shows that any special bordism extends to a constant one.  Thus special correspondences with \(\Midd = \total{\Triv^\Target}\) are equivalent if and only if they become equal after a Thom modification.

  Now \((\Triv,\total{\Triv^\Target},\anchor, \bar{\Kclass})\) is the Thom modification of \(\bigl(0, \Target, \anchor, \Thom_{\Triv^\Target}^{-1}(\bar{\Kclass})\bigr)\) along~\(\Triv^\Target\).  Hence any class in \(\Bic^*(\Base,\Target)\) is represented by \((0, \Target, \anchor, \eta)\) for a unique \(\eta \in\Coh^*_\Base(\Target)\).  Thus \(\Coh^*_\Base(\Target) \cong \Bic^*(\Base,\Target)\).
\end{proof}

This result may seem rather trivial, but it is the place where many of our technical modifications of the notion of a correspondence are used.  Without the normal factorisation or without assuming the subtriviality of the vector bundle~\(\VB\) in a normally non-singular map, we could not simplify cycles for \(\Bic^*(\Base,\Target)\) as above.  Furthermore, Theorem~\ref{the:geometric_K} is the one case where we use the definition of \(\Bic^*(\Base,\Target)\) to compute the theory.  Our proof that bivariant topological and analytic \(\K\)\nb-theory are equal will use duality to reduce the general case to this special case.  The duality argument only uses formal properties of the correspondence category and some rather special correspondences which are needed to generate the duality isomorphisms (the latter require an additional, geometric hypothesis).

\subsection{Composition of correspondences}
\label{sec:compose_correspondences}

We first define the composition only for special correspondences.  Let \(\Cor_1 \defeq (\Triv_1,\Midd_1,\mapl_1,\Kclass_1)\) and \(\Cor_2 \defeq (\Triv_2,\Midd_2,\mapl_2,\Kclass_2)\) be special correspondences from~\(\Source\) to~\(\Target\) and from~\(\Target\) to~\(\Third\), respectively.  Their composition product \(\Cor_1 \inpro \Cor_2\) is a special correspondence \((\Triv,\Midd,\mapl,\Kclass)\) from~\(\Source\) to~\(\Third\).

We let \(\Triv\defeq \Triv_1\oplus\Triv_2\) and form \(\Midd\defeq \Midd_1\times_\Target\Midd_2\) using the maps \(\Midd_1\subseteq \total{\Triv_1^\Target}\epi \Target\) and \(\mapl_2\colon \Midd_2\to\Target\).  We identify~\(\Midd\) with an open subset of~\(\total{\Triv^\Third}\) as follows:
\begin{multline*}
  \Midd = \Midd_1\times_\Target \Midd_2
  \cong \bigl\{(\third,\triv_1,\triv_2)\in \Third\times_\Base \Triv_1\times_\Base \Triv_2 = \total{\Triv^\Third} \bigm|
  \\\text{\((\third,\triv_2)\in \Midd_2\) and \((\mapl_2(\third,\triv_2),\triv_1)\in \Midd_1\)} \bigr\}
\end{multline*}
We define \(\mapl\colon \Midd\to\Source\) by \(\mapl(\midd_1,\midd_2)\defeq \mapl_1(\midd_1)\) and let \(\Kclass \defeq \Kclass_1\otimes_\Target\Kclass_2 \in\Coh^0_\Source(\Midd)\) be the exterior product of \(\Kclass_1\in\Coh^0_\Source(\Midd_1)\) and \(\Kclass_2\in\Coh^0_\Target(\Midd_2)\); more precisely, \(\Kclass_1\otimes_\Target\Kclass_2 \in \Coh^0_\Source(\Midd)\) denotes the restriction of the exterior product \(\Kclass_1\times_\Base\Kclass_2\) in \(\Coh^0_{\Source\times_\Base\Target} (\Midd_1\times_\Base\Midd_2)\) to \(\Midd_1\times_\Target\Midd_2\); we may change the support condition because \(\Source\times_\Base\Target\)\nb-compact subsets of \(\Midd_1\times_\Target\Midd_2\) are \(\Source\)\nb-compact.

This yields a special correspondence \((\Triv,\Midd,\mapl,\Kclass)\), which we denote by \(\Cor_1 \inpro \Cor_2\) or \(\Cor_1 \inpro_\Target \Cor_2\) and call the \emph{composition product} of \(\Cor_1\) and~\(\Cor_2\).

Our product construction applies equally well to special bordisms, so that products of specially bordant special correspondences remain specially bordant.  The degree is additive for products, and our product commutes with Thom modifications by \(\Grd\)\nb-vector bundles over~\(\Base\) on the first or second factor.  As a result, Theorem~\ref{the:special_bordism} shows that we we get a grading-preserving, bi-additive map
\[
\inpro_\Target=\inpro\colon
\Bic^*(\Source,\Target)\times \Bic^*(\Target,\Third) \to \Bic^*(\Source,\Third).
\]

\begin{lemma}
  \label{lem:product_special_correspondence_associative}
  The product map~\(\inpro\) is associative and turns \(\Bic^*\) into a \(\Z\)\nb-graded additive category.
\end{lemma}

\begin{proof}
  The associativity of~\(\inpro\) is routine to check.  We get a category because we also have identity correspondences.  The morphism spaces are \(\Z\)\nb-graded Abelian groups.  For an additive category, we also need products and a zero object.  It is easy to see that the empty \(\Grd\)\nb-space is a zero object and that \(\Target_1\sqcup \Target_2\) is both a coproduct and a product of \(\Target_1\) and~\(\Target_2\) in the category~\(\Bic^*\).
\end{proof}

The \emph{exterior product} of two correspondences is defined by applying~\(\times_\Base\) to all ingredients.  Exterior products of special correspondences remain special.

\begin{theorem}
  \label{the:compose_correspondence}
  With the composition product and exterior product defined above, \(\Bic^*\) becomes a \(\Z\)\nb-graded symmetric monoidal additive category; the unit object is~\(\Base\).
\end{theorem}

\begin{proof}
  This is an analogue of \cite{Emerson-Meyer:Normal_maps}*{Proposition 4.26}, which is just as trivial to prove.  It is routine to check that the exterior product is functorial for intersection products of special correspondences.  It is associative and commutative and has unit object~\(\Base\), up to certain natural homeomorphisms; these are natural with respect to ordinary maps and normally non-singular maps and hence natural with respect to correspondences.  Thus~\(\Bic^*\) is a symmetric monoidal category (see~\cite{Saavedra:Tannakiennes}).
\end{proof}

Recall that a proper \(\Grd\)\nb-map \(\mapl\colon \Target\to\Source\) yields a correspondence~\(\mapl^*\) from~\(\Source\) to~\(\Target\) by Example~\ref{exa:pullback_correspondence}.

\begin{lemma}
  \label{lem:compose_proper}
  The map \(\mapl\mapsto\mapl^*\) is a contravariant, symmetric monoidal functor from the category of \(\Grd\)\nb-spaces with proper \(\Grd\)\nb-maps to the category~\(\Bic^0\) of maps in~\(\Bic^*\) of degree~\(0\).
\end{lemma}

\begin{proof}
  Functoriality means that \(\mapl_1^* \inpro \mapl_2^* = (\mapl_1\circ\mapl_2)^*\) for proper \(\Grd\)\nb-maps \(\mapl_1\colon \Target\to\Source\) and \(\mapl_2\colon \Third\to\Target\).  Being symmetric monoidal means that if \(\mapl_1\colon \Target_1\to\Source_1\) and \(\mapl_2\colon \Target_2\to\Source_2\) are proper \(\Grd\)\nb-maps, then \(\mapl_1^*\times_\Base \mapl_2^* = (\mapl_1\times_\Base\mapl_2)^*\).  Both statements are obvious because all correspondences involved are special.
\end{proof}

More generally, consider pairs \((\mapl,\Kclass)\) where \(\mapl\colon \Target\to\Source\) is a \(\Grd\)\nb-map and \(\Kclass\in\Coh^*_\Source(\Target)\).  This becomes a special correspondence with \(\mapr=\Id\); denote this correspondence by \((\mapl,\Kclass)^*\).  The composition of such correspondences involves composing the maps~\(\mapl\) and taking an exterior product of the cohomology classes.  More precisely, let \(\mapl_1\colon \Target\to\Source\), \(\mapl_2\colon \Third\to\Target\), \(\Kclass_1\in\Coh^*_\Source(\Target)\), \(\Kclass_2\in\Coh^*_\Target(\Third)\), then
\[
(\mapl_1,\Kclass_1)^* \inpro (\mapl_2,\Kclass_2)^*
= (\mapl_1\circ\mapl_2, \Kclass_1\otimes_\Target\Kclass_2)^*,
\]
where \(\Kclass_1\otimes_\Target\Kclass_2 \in \Coh^*_\Source(\Third)\) denotes the restriction of the exterior product \(\Kclass_1\times_\Base\Kclass_2\) in \(\Coh^*_{\Source\times_\Base\Target} (\Target\times_\Base\Third)\) to the graph of \(\mapl_2\colon \Third\to \Target\).

A normally non-singular \(\Grd\)\nb-map \(\mapr\colon \Source\to\Target\) yields a correspondence~\(\mapr!\) from~\(\Source\) to~\(\Target\) by Example~\ref{exa:wrong-way_correspondence}.  We claim that this is a grading-preserving, symmetric monoidal functor from the category of normally non-singular maps to the category of correspondences, that is, it is compatible with products and exterior products.  Compatibility with exterior products and degrees is trivial, and compatibility with products is trivial for special normally non-singular submersions.  To prove functoriality for all normally non-singular maps, we need to know when products of non-special correspondences are given by an intersection product recipe.  This requires a notion of transversality.

Let \(\Cor_1 \defeq (\Midd_1,\mapl_1,\mapr_1,\Kclass_1)\) and \(\Cor_2\defeq (\Midd_2,\mapl_2,\mapr_2,\Kclass_2)\) be correspondences from~\(\Source\) to~\(\Target\) and from~\(\Target\) to~\(\Third\).  Write \(\mapr_1 = (\VB_1,\Triv_1,\hat{\mapr}_1)\) and \(\mapr_2 = (\VB_2,\Triv_2,\hat{\mapr}_2)\).  Let
\begin{align*}
  \Midd&\defeq \Midd_1\times_\Target\Midd_2
  = \{(\midd_1,\midd_2)\in\Midd_1\times\Midd_2\mid
  \mapr_1(\midd_1) = \mapl_2(\midd_2)\},\\
  \mapl&\colon \Midd\to\Triv,\qquad
  (\midd_1,\midd_2)\mapsto\mapl_1(\midd_1),\\
  \Kclass&\defeq \Kclass_1\otimes_\Target\Kclass_2
  \qquad \text{in \(\Coh^0_\Source(\Midd)\);}
\end{align*}
here~\(\mapr_1\) also denotes the trace of the normally non-singular map~\(\mapr_1\).  To get a correspondence from~\(\Source\) to~\(\Other\), we need a normally non-singular map \(\mapr=(\Triv,\hat{\mapr},\VB)\) from~\(\Midd\) to~\(\Third\); its trace should be the product of the coordinate projection \(\Midd\to\Midd_2\) with the trace of~\(\mapr_2\).  We put
\[
\Triv \defeq \Triv_1\oplus\Triv_2,
\qquad
\VB \defeq \pr_1^*(\VB_1)\oplus\pr_2^*(\VB_2)
\]
with the induced \(\Coh\)\nb-orientations.  In general, there need not be an open embedding \(\total{\VB}\opem\total{\Triv^\Third}\): this is where we need transversality.

\begin{definition}
  \label{def:transverse}
  Let \(\Midd_1\), \(\Midd_2\), and~\(\Target\) be \(\Grd\)\nb-spaces, let \(\mapr= (\VB_1,\Triv_1,\hat{\mapr}_1)\) be a normally non-singular \(\Grd\)\nb-map from~\(\Midd_1\) to~\(\Target\), and let \(\mapl_2\colon \Midd_2\to\Target\) be a \(\Grd\)\nb-map.  View~\(\total{\VB_1}\) as a space over~\(\Target\) via \(\total{\VB_1}\subseteq\total{\Triv_1^\Target}\epi\Target\).  Let \(\VB_1^\Midd\) be the pull-back of~\(\VB_1\) to~\(\Midd \defeq \Midd_1\times_\Target \Midd_2\) along the canonical projection \(\pr_1\colon \Midd\to \Midd_1\).

  We call~\(\mapr_1\) \emph{transverse} to~\(\mapl_2\) if the map
  \[
  \zers{\VB_1}\times_\Target\Id_{\Midd_2}\colon
  \Midd\defeq \Midd_1\times_\Target\Midd_2 \to
  \total{\VB_1}\times_\Target\Midd_2
  \]
  extends to an open embedding from~\(\total{\VB_1^\Midd}\) into \(\total{\VB_1}\times_\Target\Midd_2\).

  Two correspondences \(\Cor_1\) and~\(\Cor_2\) as above are \emph{transverse} if~\(\mapr_1\) is transverse to~\(\mapl_2\).
\end{definition}

More precisely, a transverse pair of maps is a triple consisting of a map~\(\mapl_2\), a normally non-singular map~\(\mapr_1 = (\VB_1, \Triv_1, \hat{\mapr}_1)\), and an open embedding from~\(\total{\VB_1^\Midd}\) into \(\total{\VB_1}\times_\Target\Midd_2\).  A transverse pair of correspondences is a similar triple.

The total space of~\(\VB_1^\Midd\) is
\[
\total{\VB_1^\Midd} = \bigl\{(\vb_1,\midd_2) \in
\total{\VB_1}\times\Midd_2
\bigm| \mapr_1\circ \proj{\VB_1}(\vb_1) = \mapl_2(\midd_2)
\bigr\}.
\]
This may differ drastically from
\[
\total{\VB_1}\times_\Target\Midd_2 =
\bigl\{(\vb_1,\midd_2) \in \total{\VB_1}\times \Midd_2 \bigm|
\proj{\Triv_1^\Target}\circ \hat{\mapr}_1(\vb_1)
= \mapl_2(\midd_2) \bigr\}.
\]

\begin{remark}
  \label{rem:transverse_and_normal_embedding}
  The transversality condition Definition~\ref{def:transverse} asserts that the embedding \(\Midd_1\times_\Target \Midd_2 \to \total{\VB_1}\times_\Target \Midd_2\) be a \emph{normally non-singular} embedding and that its normal bundle be \(\pr_1^*(\VB_1)\).
\end{remark}

Before we check that transversality ensures that the intersection product of two correspondences exists and represents their product, we compare it to the usual notion of transversality for smooth maps.

\begin{example}
  \label{exa:transverse_smooth}
  Let \(\Midd_1\), \(\Midd_2\), and~\(\Target\) be smooth manifolds, let \(\mapl\colon \Midd_2\to\Target\) be a smooth map, and let \(\mapr=\bigl(\VB,\R^n,\hat{\mapr}\bigr)\) be the lifting of a smooth map \(\varphi\colon \Midd_1\to\Target\) using a smooth embedding \(h\colon \Midd_1\to\R^n\).  Thus~\(\VB\) is the normal bundle of the embedding \((\varphi,h)\circ \zers{\VB}\colon \Midd_1\to \Target\times\R^n\).  Assume that the maps \(\mapl\) and~\(\varphi\) are transverse in the usual sense that
  \[
  D_{\midd_1}\varphi(\Tvert_{\midd_1}\Midd_1) +
  D_{\midd_2}\mapl(\Tvert_{\midd_2}\Midd_2)
  = \Tvert_\target\Target
  \]
  for all \(\midd_1\in\Midd_1\), \(\midd_2\in\Midd_2\) with \(\target\defeq \varphi(\midd_1) = \mapl(\midd_2)\).  Then \(\Midd_1\times_\Target\Midd_2\) is a smooth submanifold of \(\Midd_1\times\Midd_2\) and hence a smooth manifold.  Since the map \(\total{\VB}\subseteq\total{\Triv^\Target}\epi \Target\) is a submersion, \(\total{\VB}\times_\Target\Midd_2\) is a smooth manifold as well.  The map \(\zers{\VB}\times_\Target \Id_{\Midd_2}\) is a smooth embedding because it is the restriction of the smooth embedding \(\zers{\VB}\times\Id_{\Midd_2}\).  We claim that its normal bundle is the pull-back of~\(\VB\) to \(\Midd_1\times_\Target\Midd_2\).  This follows from the vector bundle isomorphisms
  \begin{align*}
    \Tvert\Midd = \Tvert (\Midd_1\times_\Target \Midd_2)
    &\cong \pr_1^*(\Tvert \Midd_1)\oplus_{(\pr_1 \circ \mapr_1)^*(\Tvert \Target)} \pr_2^*(\Tvert\Midd_2),\\
    \Tvert(\total{\VB}\times_\Target \Midd_2)
    &\cong \pr_1^*(\Tvert \total{\VB})\oplus_{\Tvert\Target} \pr_2^*(\Tvert\Midd_2)\\
    \zers{\VB}^*(\Tvert \total{\VB}) &\cong \Tvert \Midd_1\oplus \VB,
  \end{align*}
  which combine to show that the cokernel of the vector bundle map \(\Tvert (\Midd) \to \Tvert (\total{\VB}\times_\Target \Midd_2)\) is \(\pr_1^*(\VB)\).  Thus \(\zers{\VB}\times_\Target \Id_{\Midd_2}\) is a smooth normally non-singular embedding with the required normal bundle (compare Remark~\ref{rem:transverse_and_normal_embedding}) so that \(\mapl\) and~\(\varphi\) are transverse in the sense of Definition~\ref{def:transverse}.
\end{example}

Now we return to the problem of computing the product of two non-special correspondences \(\Cor_1\) and~\(\Cor_2\).  We follow our previous notation and, in particular, define \(\Midd\), \(\mapl\colon \Midd\to\Source\), \(\Kclass\in\Coh^*_\Source(\Midd)\), and the \(\Grd\)\nb-vector bundles \(\VB\) and~\(\Triv\) as above.

The total space of~\(\VB\) is \(\total{\VB_1}\times_\Target \total{\VB_2}\), which agrees with the total space of the pull-back of~\(\VB_2\) to \(\total{\VB_1}\times_\Target\Midd_2\).  Note that to form \(\total{\VB_1}\times_\Target \total{\VB_2}\) it makes no difference which map from~\(\total{\VB_1}\) to~\(\Target\) we use: \(\mapr_1\circ \pi_{\VB_1}\) or \(\pi_{\Triv^\Target}\circ \hat{\mapr}\); \cite{Emerson-Meyer:Normal_maps}*{Proposition 2.22} shows that only the homotopy class of the map \(\total{\VB_1}\to\Target\) matters.

If our correspondences are transverse, there are normally non-singular embeddings
\[
\Midd
\to \total{\VB_1}\times_\Target\Midd_2
\to \total{\VB_1}\times_\Target\total{\VB_2}
\]
with normal bundles \(\pi_1^*\VB_1\) (denoted \(\VB_1^M\) in Definition \ref{def:transverse}) and \(\pi_2^*\VB_2\), respectively; the first normally non-singular map is the transversality assumption, the second one is obvious because~\(\VB_2\) is a \(\Grd\)\nb-vector bundle over~\(\Midd_2\).  Composition yields a normally non-singular embedding \(\Midd\to\total{\VB_1}\times_\Target\total{\VB_2}\) with normal bundle~\(\VB\), that is, an open embedding from~\(\total{\VB}\) into \(\total{\VB_1}\times_\Target\total{\VB_2}\).

When we first replace our two correspondences by special ones and then take their intersection product, we replace \(\Midd_1\) by~\(\total{\VB_1}\) and \(\Midd_2\) by~\(\total{\VB_2}\) and construct a special normally non-singular submersion \(\total{\VB_1}\times_\Target\total{\VB_2} \opem \total{\Triv^\Third} \epi\Third\); we compose the open embedding \(\total{\VB_1}\times_\Target\total{\VB_2} \opem \total{\Triv^\Third}\) from this previous construction with the open embedding \(\total{\VB}\opem\total{\VB_1}\times_\Target\total{\VB_2}\) from transversality to get an open embedding \(\hat{\mapr}\colon \total{\VB}\opem\total{\Triv^\Third}\).

This produces the desired \(\Coh\)\nb-oriented normally non-singular map \((\VB,\hat{\mapr},\Triv)\) from~\(\Midd\) to~\(\Third\) and hence a correspondence from~\(\Source\) to~\(\Third\), called the \emph{intersection product} of \(\Cor_1\) and~\(\Cor_2\) and denoted \(\Cor_1 \inpro \Cor_2\) or \(\Cor_1 \inpro_\Target \Cor_2\).  This is only defined if \(\Cor_1\) and~\(\Cor_2\) are transverse and, at first sight, depends on the choice of the open embedding in the definition of transversality.

\begin{theorem}
  \label{the:compose_transverse_correspondence}
  If the correspondences \(\Cor_1\) and~\(\Cor_2\) are transverse, then their intersection product is equivalent to the composition product of the equivalent special correspondences.
\end{theorem}

\begin{proof}
  The Thom modification of \(\Cor_1 \inpro \Cor_2\) along~\(\VB\) is a special correspondence \((\Triv, \total{\VB}, \mapl, \Kclass)\) from~\(\Source\) to~\(\Other\), where we use~\(\hat{\mapr}\) to view~\(\total{\VB}\) as an open subset of~\(\total{\Triv^\Target}\) and let~\(\mapl\) be the composition of the bundle projection \(\total{\VB}\epi\Midd\) with the map \(\Midd\to\Midd_1\to\Source\).  When we first Thom modify \(\Cor_1\) and~\(\Cor_2\) along~\(\VB_1\) and~\(\VB_2\) to make them special and then compose, we get a special correspondence \((\Triv, \VB', \mapl', \Kclass')\) with \(\VB' = \total{\VB_1}\times_\Target \total{\VB_2}\); we have seen that this contains~\(\total{\VB}\) as an open subset.  The map~\(\mapl'\) extends~\(\mapl\) on~\(\total{\VB}\), and~\(\Kclass'\) extends~\(\Kclass\).  Hence the two correspondences via \(\total{\VB}\) and~\(\VB'\) are specially bordant by Example~\ref{exa:bordism_extend_from_open}.
\end{proof}

\begin{example}
  \label{exa:compose_one_split}
  A pair of correspondences \(\Cor_1\) and~\(\Cor_2\) is transverse if~\(\Cor_1\) is special, regardless of~\(\Cor_2\), for \(\VB_1\) is the \(0\)-vector bundle in this case, making the condition in Definition \ref{def:transverse} trivially satisfied.
\end{example}

\begin{example}
  \label{exa:transverse_Mid2}
  If \(\mapl_2\colon \Midd_2\epi\Target\) is a vector bundle projection, then \(\Cor_1\) and~\(\Cor_2\) are transverse.  To see this, note that \(\Midd_1\times_\Target\Midd_2\) and \(\total{\VB_1}\times_\Target\Midd_2\) are the total spaces of the pull-backs of the vector bundle~\(\Midd_2\) to \(\Midd_1\) and to~\(\total{\VB_1}\) under the maps \(\mapr \colon \Midd_1 \to \Target\) and \(\pi_{\Triv^\Target}\circ \hat{\mapr}\colon \VB_1\to \Target\).  Now the maps \(\pi_{\Triv^\Target}\circ \hat{\mapr}\) and \(\mapr\circ \pi_{\VB_1}\) are homotopic by a homotopy which is constant on the zero section.  Hence the corresponding pull-backs of~\(M_2\) are isomorphic via an isomorphism which is the identity on the zero section, and, in particular, \(\total{\VB_1}\times_\Target\Midd_2\) is homeomorphic to~\(\total{\VB_1^M}\) via a homeomorphism which is the identity on the zero section (see \cite{Emerson-Meyer:Normal_maps}*{Proposition 2.22}).  Therefore we are left with a vector bundle \(W \defeq \mapr_1^*(\Midd_2)\) on~\(\Midd_1\).  It is obvious that the zero-section embedding \(W\to \pi_{\VB_1}^*(W)\) is normally non-singular with normal bundle \(\pi_{\VB_1}^*(W)\defeq (\mapr\circ \pi_{\VB_1})^*(M_2)\) as required.
\end{example}

\begin{example}
  \label{exa:split_correspondence}
  Let \((\Midd,\mapl,\mapr,\Kclass)\) be a correspondence from~\(\Source\) to~\(\Target\).  Then \((\mapl,\Kclass)^*\) is a correspondence from~\(\Source\) to~\(\Midd\) and~\(\mapr!\) is a correspondence from~\(\Midd\) to~\(\Target\).  These two are transverse by Example~\ref{exa:transverse_Mid2}, and their composition product \((\mapl,\Kclass)^* \inpro \mapr!\) is the given correspondence \((\Midd,\mapl,\mapr,\Kclass)\).
\end{example}

Finally, we use Example~\ref{exa:transverse_Mid2} to show that the composition of correspondences generalises the composition of normally non-singular maps:

\begin{corollary}
  \label{cor:compose_normal_cor}
  Let \(\mapr_1\colon \Source\to\Target\) and \(\mapr_2\colon \Target\to\Third\) be \(\Coh\)\nb-oriented normally non-singular maps.  Then \(\mapr_1!  \inpro \mapr_2!  = (\mapr_2\circ\mapr_1)!\).
\end{corollary}

\begin{proof}
  Let~\(\Cor_j\) be the canonical representatives for~\(\mapr_j!\) for \(j=1,2\).  We conclude that \(\Midd=\Source\), \(\mapl=\Id\), \(\Kclass=1\), and the transversality condition is automatic by Example~\ref{exa:transverse_Mid2}.  Hence the product is represented by the intersection product \(\Cor_1 \inpro \Cor_2\) by Theorem~\ref{the:compose_transverse_correspondence}.  This is of the form~\(\mapr!\) for a normally non-singular map \(\mapr\colon \Source\to\nobreak\Third\).  Inspection shows that~\(\mapr\) agrees with the product of normally non-singular maps~\(\mapr_2\circ\mapr_1\).
\end{proof}

As in Kasparov theory, we may combine exterior products and composition products to an operation
\begin{multline}
  \label{eq:gen_intersection_product}
  \inpro_\Third\colon
  \Bic^i(\Source_1,\Target_1\times_\Base\Third) \times
  \Bic^j(\Third\times_\Base\Source_2,\Target_2) \to
  \Bic^{i+j}(\Source_1\times_\Base\Source_2,
  \Target_1\times_\Base\Target_2),\\
  (\alpha,\beta) \mapsto
  (\alpha\times_\Base\Id_{\Source_2}) \inpro
  (\Id_{\Target_1} \times_\Base\beta),
\end{multline}
which is again associative and graded commutative in a suitable sense.  This operation will be used heavily in~\S\ref{sec:duality} to construct duality isomorphisms.

Recall that \(\Bic^*(\Base,\Target) \cong \Coh^*_\Base(\Target)\) is the \(\Coh\)\nb-cohomology of~\(\Target\) with \(\Base\)\nb-compact support (Theorem~\ref{the:geometric_K}).  Since~\(\Bic^*\) is a category, correspondences act on \(\Coh^*_\Base(\Target)\); this extends the wrong-way maps for \(\Coh\)\nb-oriented normally non-singular maps in \cite{Emerson-Meyer:Normal_maps}*{Theorem 5.16}.

\subsection{Composition of smooth correspondences using transversality}
\label{sec:compose_transversal}

We recover the transversality formula for the composition of two smooth correspondences in general position of Connes and Skandalis~\cite{Connes-Skandalis:Indice_feuilletages}.

\begin{theorem}
  \label{the:transverse_smooth}
  Let \(\Midd_1\), \(\Midd_2\), and~\(\Target\) be smooth \(\Grd\)\nb-manifolds; let \(\NM = (\VB,\Triv,\hat{\mapr})\) be a smooth normally non-singular \(\Grd\)\nb-map from~\(\Midd_1\) to~\(\Target\) with trace~\(\mapr\) and let \(\mapl\colon \Midd_2\to\Target\) be a smooth \(\Grd\)\nb-map.  These two maps are transverse if
  \[
  D_{\midd_1}\mapr(\Tvert_{\midd_1}\Midd_1) +
  D_{\midd_2}\mapl(\Tvert_{\midd_2}\Midd_2)
  = \Tvert_\target\Target
  \]
  for all \(\midd_1\in\Midd_1\), \(\midd_2\in\Midd_2\) with \(\target\defeq \mapr(\midd_1) = \mapl(\midd_2)\).
\end{theorem}

\begin{proof}
  The argument is literally the same as in the non-equivariant case, see Example~\ref{exa:transverse_smooth}.  The Tubular Neighbourhood we need exists by \cite{Emerson-Meyer:Normal_maps}*{Theorem 3.18}.
\end{proof}

\begin{corollary}
  \label{cor:compose_smooth_normal_correspondences}
  Let \(\NM_1 = (\Midd_1,\mapl_1,\mapr_1,\Kclass_1)\) and \(\NM_2 = (\Midd_2,\mapl_2,\mapr_2,\Kclass_2)\) be smooth correspondences from~\(\Source\) to~\(\Target\) and from~\(\Target\) to~\(\Third\), respectively.  Assume that both \(\Midd_1\) and~\(\Midd_2\) admit smooth normally non-singular maps to~\(\Base\), so that we lose nothing if we view \(\mapr_1\) and~\(\mapr_2\) as \(\Coh\)\nb-oriented smooth maps.  Assume also that \(\mapr_1\) and~\(\mapl_2\) are transverse as in Theorem~\textup{\ref{the:transverse_smooth}}.  Then \(\Midd_1\times_\Target\Midd_2\) is a smooth \(\Grd\)\nb-manifold with a smooth normally non-singular map to~\(\Base\) as well.  The intersection product of the two correspondences above is
  \[
  \NM_1 \inpro_\Target \NM_2
  = \bigl(\Midd_1\times_\Target\Midd_2, \mapl_1\circ\pi_1,
  \mapr_2\circ\pi_2, \pi_1^*(\Kclass_1)\otimes
  \pi_2^*(\Kclass_2)\bigr),
  \]
  where \(\pi_j\colon \Midd_1\times_\Target \Midd_2\to\Midd_j\) for \(j=1,2\) are the canonical projections.
\end{corollary}

\begin{proof}
  If \(\Midd_1\) and~\(\Midd_2\) admit smooth normally non-singular maps to~\(\Base\), then so does \(\Midd_1\times_\Target\Midd_2\) because it embeds in \(\Midd_1\times_\Base \Midd_2\), which embeds in \(\total{\Triv_1 \oplus \Triv_2} = \total{\Triv_1} \times_\Base \total{\Triv_2}\) if~\(\Midd_j\) embeds in~\(\total{\Triv_j}\) for \(j=1,2\).  Under this assumption, we can replace all smooth normally non-singular maps by mere smooth maps, so that it suffices to describe the traces and the \(\Coh\)\nb-orientations of the normally non-singular maps we are dealing with.  Hence the assertion follows from the construction of the intersection product for transverse correspondences in~\S\ref{sec:compose_correspondences}.  We leave it to the reader to write down the \(\Coh\)\nb-orientation that the map \(\mapr_2\circ\pi_2\) inherits.
\end{proof}

\section{Duality isomorphisms}
\label{sec:duality}

There is a canonical notion of duality in~\(\Bic^*\) because it is a symmetric monoidal category: two \(\Grd\)\nb-spaces \(\Tot\) and~\(\Dual\) are \emph{dual} in~\(\Bic^*\) if there is a natural isomorphism
\[
\Bic^*(\Tot\times_\Base\Third_1,\Third_2) \cong
\Bic^*(\Third_1,\Dual\times_\Base\Third_2)
\]
for all \(\Grd\)\nb-spaces \(\Third_1\) and~\(\Third_2\).  This is equivalent to the symmetric duality isomorphism
\[
\Bic^*(\Dual\times_\Base\Third_1,\Third_2) \cong
\Bic^*(\Third_1,\Tot\times_\Base\Third_2).
\]

But this notion does not cover familiar duality isomorphisms for non-compact smooth manifolds.  Let~\(\Tot\) be a smooth manifold and let~\(\Tvert\Tot\) be its tangent space.  Then there are natural isomorphisms
\[
\RK^*(\Tot)\cong \K_*(\Tvert\Tot),\qquad
\K_*(\Tot)\cong \K^*_\Tot(\Tvert\Tot)
\]
between the representable \(\K\)\nb-theory of~\(\Tot\) and the \(\K\)\nb-homology of~\(\Tvert\Tot\), and between the \(\K\)\nb-homology of~\(\Tot\) and the \(\K\)\nb-theory of~\(\Tvert\Tot\) with \(\Tot\)\nb-compact support.  These two duality isomorphisms are generalised in~\cite{Emerson-Meyer:Dualities}, following Gennadi Kasparov~\cite{Kasparov:Novikov}.  The abstract conditions in~\cite{Emerson-Meyer:Dualities} that are equivalent to the existence of such duality isomorphisms only use formal properties of equivariant Kasparov theory and therefore carry over to the geometric setting we consider here.  We sketch this generalisation in this section.

The two duality isomorphisms have the following important applications.  The first duality is used in~\cite{Emerson-Meyer:Dualities} to define equivariant Euler characteristics and equivariant Lefschetz invariants in equivariant \(\K\)\nb-homology; the geometric counterpart of this duality described below will be used in a forthcoming article to compute Euler characteristics and Lefschetz invariants of correspondences in geometric terms.  The second duality allows, in particular, to reduce bivariant \(\K\)\nb-groups to \(\K\)\nb-theory with support conditions.  This will be used below to show that the topological and analytic versions of bivariant \(\K\)\nb-theory agree if there is a duality isomorphism.

Our notion of duality does not contain Spanier--Whitehead duality as a special case.  The main issue is that we require the dual of~\(\Tot\) to be a space over~\(\Tot\).  This seems unavoidable for the second duality isomorphism and rules out taking a complement of~\(\Tot\) in some ambient space as in Spanier--Whitehead Duality.

\subsection{Preparations}
\label{sec:prepare_duality}

First we need some notation.  To emphasise the groupoid under consideration, we now write \(\Coh^*_\Grd(\Third)\) and \(\Bic^*_\Grd(\Source,\Target)\) instead of \(\Coh^*(\Third)\) and \(\Bic^*(\Source,\Target)\).  Let~\(\Third\) be a \(\Grd\)\nb-space.  Recall that \(\Grd\ltimes\Third\)-spaces are \(\Grd\)\nb-spaces with a \(\Grd\)\nb-map to~\(\Third\).  Hence a cohomology theory~\(\Coh^*_\Grd\) for \(\Grd\)\nb-spaces restricts to one for \(\Grd\ltimes\Third\)\nb-spaces.  We denote the latter by \(\Coh^*_{\Grd\ltimes\Third}\) and get a corresponding bivariant theory \(\Bic^*_{\Grd\ltimes\Third}(\Source,\Target)\) for \(\Grd\ltimes\Third\)\nb-spaces \(\Source\) and~\(\Target\).  The special correspondences \((\Triv,\Midd,\mapl,\Kclass)\) that enter in its definition differ from the ones for \(\Bic^*_\Grd(\Source,\Target)\) in the following ways:
\begin{itemize}
\item \(\Triv\) is an \(\Coh\)\nb-oriented \(\Grd\)\nb-vector bundle on~\(\Third\), not on~\(\Base\);

\item the map \(\mapl\colon \Midd\to\Source\) is a \(\Grd\)\nb-map over~\(\Third\).
\end{itemize}
The first modification is of little importance: if we make the mild assumption that any \(\Grd\)\nb-vector bundle over~\(\Third\) is subtrivial, then we may use Thom modification to reduce to \(\Grd\ltimes\Third\)-equivariant correspondences whose \(\Grd\)\nb-vector bundle over~\(\Third\) is trivial.  But the second condition has a significant effect.

Theorem~\ref{the:geometric_K} generalises to an isomorphism
\begin{equation}
  \label{eq:Bic_contains_Coh}
  \Bic^*_{\Grd\ltimes\Source}(\Source,\Target) \cong
  \Coh^*_{\Grd,\Source}(\Target)
\end{equation}
for any \(\Grd\ltimes\Source\)-space~\(\Target\), where the right hand side denotes the \(\Source\)\nb-compactly supported version of~\(\Coh^*_\Grd\).  This is quite different than \(\Bic^*_\Grd(\Source,\Target)\).

The functoriality properties of normally non-singular maps carry over to correspondences.

First, a map \(\varphi\colon \Third_1\to\Third_2\) induces a symmetric monoidal functor
\[
\varphi^*\colon
\Bic^*_{\Grd\ltimes\Third_2}(\Source,\Target) \to
\Bic^*_{\Grd\ltimes\Third_1}(\varphi^*\Source,
\varphi^*\Target).
\]
We often write \(\Third_1 \times_{\Third_2} \alpha\) instead of~\(\varphi^*\alpha\) for \(\alpha \in \Bic^*_{\Grd\ltimes\Third_2}(\Source,\Target)\).

Secondly, if all \(\Grd\ltimes\Third_2\)\nb-vector bundles over~\(\Third_1\) are subtrivial (that is, direct summands of \(\Grd\)\nb-vector bundles pulled back from~\(\Third_2\)), then there is a forgetful functor
\[
\Bic^*_{\Grd\ltimes\Third_1}(\Source,\Target) \to
\Bic^*_{\Grd\ltimes\Third_2}(\Source,\Target)
\]
in the opposite direction, where we view \(\Grd\ltimes\Third_1\)\nb-spaces as \(\Grd\ltimes\Third_2\)\nb-spaces by composing the anchor map to~\(\Third_1\) with~\(\varphi\).  We usually denote the image of \(g\in\Bic^*_{\Grd\ltimes\Third}(\Source,\Target)\) under the forgetful functor by \(\overline{g} \in \Bic^*_\Grd(\Source,\Target)\).

\begin{remark}
  \label{rem:drop_pull-back}
  When we compose morphisms, we sometimes drop pull-back functors and forgetful functors from our notation.  For instance, if \(\Theta\in\Bic^*_{\Grd\ltimes\Source}(\Source, \Source\times_\Base\Dual)\) and \(D\in\Bic^*_\Grd(\Dual,\Base)\) for two \(\Grd\)\nb-spaces \(\Source\) and~\(\Dual\), then \(\Theta \inpro_\Dual D\in \Bic^*_{\Grd\ltimes\Source}(\Source,\Source) \cong \Coh^*(\Source)\) denotes the product of~\(\Theta\) and \(\anchor^*(D) \in \Bic^*_{\Grd\ltimes\Source}(\Source\times_\Base\Dual,\Source)\) where \(\anchor\colon \Source\to\Base\) is the anchor map and we identify \(\Source\times_\Base\Base\cong\Source\).
\end{remark}

\begin{definition}
  \label{def:underline_factor}
  Let~\(\Tot\) be a \(\Grd\)\nb-space and let \(\Other_1\) and~\(\Other_2\) be two \(\Grd\)\nb-spaces over~\(\Tot\).  Then we may view \(\Other_1\times_\Base\Other_2\) as a \(\Grd\)\nb-space over~\(\Tot\) in two different ways, using the first or second coordinate projection followed by the anchor map \(\Other_j\to\Tot\).  To distinguish these two \(\Grd\ltimes\Tot\)\nb-spaces, we underline the factor whose \(\Tot\)\nb-structure is used.  Thus the groups \(\Bic^*_{\Grd\ltimes\Tot}(\Tot, \underline{\Tot}\times_\Base\Dual)\) and \(\Bic^*_{\Grd\ltimes\Tot}(\Tot, \Tot\times_\Base\underline{\Dual})\) for a \(\Grd\ltimes\Tot\)-space~\(\Dual\) are different.
\end{definition}

\subsection{The two duality isomorphisms}
\label{sec:duality_isomorphisms}

Throughout this section, \(\Tot\) is a \(\Grd\)\nb-space, \(\Dual\) is a \(\Grd\ltimes\Tot\)-space, and \(D\in\Bic^{-n}_\Grd(\Dual,\Base)\) for some \(n\in\Z\); \(\Third\) is a \(\Grd\ltimes\Tot\)\nb-space, and~\(\Target\) is a \(\Grd\)\nb-space.  We assume throughout that all \(\Grd\)\nb-vector bundles over~\(\Source\) are subtrivial.  We are going to define two duality maps involving this data and then analyse when they are invertible, following~\cite{Emerson-Meyer:Dualities}.

The \emph{first duality map} for \((\Tot,\Dual,D)\) with coefficients \(\Third\) and~\(\Target\) is the map
\begin{equation}
  \label{eq:first_duality}
  \PD^*\colon \Bic^i_{\Grd\ltimes\Tot}(\Third,
  \Tot\times_\Base\Target)\to
  \Bic^{i-n}_\Grd(\Dual\times_\Tot\Third, \Target),\qquad
  g\mapsto (-1)^{in} \overline{(\Dual\times_\Tot g)} \inpro_\Dual D.
\end{equation}
The \emph{second duality map} for \((\Tot,\Dual,D)\) with coefficients \(\Third\) and~\(\Target\) is the map
\begin{equation}
  \label{eq:second_duality}
  \SPD^*\colon \Bic_{\Grd\ltimes\Tot}^i
  (\Third,\Dual\times_\Base\Target) \to
  \Bic_\Grd^{i-n}(\Third,\Target),\qquad
  f\mapsto (-1)^{in} \overline{f}\inpro_\Dual D.
\end{equation}
In both cases, the overlines denote the forgetful functor \(\Bic^*_{\Grd\ltimes\Tot}\to\Bic^*_\Grd\).

The second duality map is particularly interesting for \(\Third=\Tot\): then it maps \(\Coh^i_{\Grd,\Tot}(\Dual\times_\Base\Target) \cong \Bic_{\Grd\ltimes\Tot}^i(\Tot,\Dual\times_\Base\Target)\) to \(\Bic_\Grd^{i-n}(\Tot,\Target)\) by Theorem~\ref{the:geometric_K}.

Necessary and sufficient conditions for analogous duality maps in Kasparov theory to be isomorphisms are analysed in~\cite{Emerson-Meyer:Dualities}.  These carry over literally to our setting because they only use formal properties of Kasparov theory.

\begin{theorem}
  \label{the:first_duality}
  Fix \(\Tot\), \(\Dual\), \(D\) and~\(\Third\).  The first duality map is an isomorphism for all \(\Grd\)\nb-spaces~\(\Target\) if and only if there is \(\Theta_\Third \in \Bic^n_{\Grd\ltimes\Tot}\bigl(\Third, \underline{\Tot}\times_\Base (\Dual\times_\Tot\Third)\bigr)\) with the following properties:
  \begin{enumerate}[label=\textup{(\roman{*})}]
  \item \label{first_duality_1} \(\overline{(\Dual\times_\Tot \Theta_\Third)} \inpro_\Dual D = (-1)^n \Id_{\Dual\times_\Tot\Third}\) in \(\Bic^0_\Grd(\Dual\times_\Tot\Third, \Dual\times_\Tot\Third)\);

  \item \label{first_duality_2} \((-1)^{in} \Theta_\Third \inpro_{\Dual\times_\Tot\Third} \overline{(\Dual\times_\Tot g)} \inpro_\Dual D = g\) for all \(g\in \Bic^i_{\Grd\ltimes\Tot}(\Third,\Tot\times_\Base \Target)\) and all \(\Grd\)\nb-spaces~\(\Target\).
  \end{enumerate}
  Furthermore, the inverse of\/ \(\PD^*\) is of the form
  \begin{equation}
    \label{eq:first_duality_inverse}
    \PD\colon \Bic^{i-n}_\Grd(\Dual\times_\Tot\Third,\Target)
    \to \Bic^i_{\Grd\ltimes\Tot}(\Third,
    \Tot\times_\Base\Target),\qquad
    f\mapsto
    \Theta_\Third \inpro_{\Dual\times_\Tot\Third} f,
  \end{equation}
  and \(\Theta_\Third\) is determined uniquely.

  Suppose that \(\Theta\in\Bic^n_{\Grd\times\Tot}(\Tot, \underline{\Tot}\times_\Base\Dual)\) satisfies \(\Theta\inpro_\Dual D=\Id_\Tot\) in \(\Bic^n_{\Grd\times\Tot}(\Tot, \Tot)\).  Then the following conditions \ref{first_duality_3} and~\ref{first_duality_4} imply \ref{first_duality_1} and~\ref{first_duality_2}:
  \begin{enumerate}[resume, label=\textup{(\roman{*})}]
  \item \label{first_duality_3} the following diagram in~\(\Bic^*_\Grd\) commutes:
    \[
    \xymatrix{
      \Dual\times_\Tot \Third
      \ar[r]^-{\overline{\Dual\times_\Tot\Theta_\Third}}
      \ar[d]_{\overline{(\Dual\times_\Tot\Third) \times_\Tot \Theta}}&
      \Dual\times_\Base (\Dual\times_\Tot\Third)
      \ar@{<->}[dl]_\cong^{(-1)^n\flip}\\
      (\Dual\times_\Tot\Third) \times_\Base \Dual.
    }
    \]

  \item \label{first_duality_4} \(\Theta_\Third \inpro_{\Dual\times_\Tot\Third} \overline{(\Dual\times_\Tot g)} = \Theta \inpro_\Tot g\) in \(\Bic^{i+n}_{\Grd\ltimes\Tot}(\Third, \underline{\Tot}\times_\Base \Dual\times_\Base\Target)\) for all \(g\in \Bic^i_{\Grd\ltimes\Tot}(\Third, \Tot\times_\Base \Target)\) and all \(\Grd\)\nb-spaces~\(\Target\).
  \end{enumerate}
\end{theorem}

\begin{proof}
  Condition~\ref{first_duality_1} means that \(\Theta_\Third\in \Bic^n_{\Grd\ltimes\Tot}(\Third,\underline{\Tot}\times_\Base \Dual\times_\Tot\Third)\) satisfies \(\PD^*(\Theta_\Third) = \Id_{\Dual\times_\Tot\Third}\).  Hence~\ref{first_duality_1} is necessary for \(\PD^*\) to be invertible and determines~\(\Theta_\Third\) uniquely.  The associativity of~\(\inpro\) and the graded commutativity of exterior products yield
  \begin{multline*}
    \PD^*\circ\PD(f)
    = (-1)^{in} \overline{(\Dual\times_\Tot\Theta_\Third)}
    \inpro_{\Dual\times_\Tot \Third} (f \inpro_\Dual D)
    \\= (-1)^n \overline{(\Dual\times_\Tot\Theta_\Third)}
    \inpro_\Dual D \inpro_{\Dual\times_\Tot \Third} f
  \end{multline*}
  for \(f\in \Bic^{i-n}_\Grd(\Dual\times_\Tot\Third,\Target)\).  Hence \(\PD^*\circ\PD\) is the identity map if and only if Condition~\ref{first_duality_1} holds.  Then the inverse of~\(\PD^*\) can only be~\(\PD\).  By definition,
  \[
  \PD\circ\PD^*(g) = (-1)^{in} \Theta_\Third
  \inpro_{\Dual\times_\Tot\Third}
  \overline{(\Dual\times_\Tot g)} \inpro_\Dual D
  \]
  for all \(g\in \Bic^i_{\Grd\ltimes\Tot} (\Third,\Tot\times_\Base \Target)\).  Hence~\ref{first_duality_2} is equivalent to \(\PD\circ\PD^*=\Id\).  As a result, the maps \(\PD\) and~\(\PD^*\) defined as in \eqref{eq:first_duality} and~\eqref{eq:first_duality_inverse} are inverse to each other if and only if Conditions \ref{first_duality_1} and~\ref{first_duality_2} hold.

  Now assume that there is a class~\(\Theta\) as above.  Condition \ref{first_duality_3} implies~\ref{first_duality_1}; and~\ref{first_duality_4} implies~\ref{first_duality_2} because of the graded commutativity of exterior products:
  \[
  (-1)^{in} \Theta_\Third \inpro_{\Dual\times_\Tot\Third}
  \overline{(\Dual\times_\Tot g)} \inpro_\Dual D
  = (-1)^{in} \Theta \inpro_\Tot g \inpro_\Dual D
  = g \inpro_\Tot \Theta  \inpro_\Dual D
  = g
  \]
  for all~\(g\).
\end{proof}

A class~\(\Theta\) as above exists and is equal to~\(\Theta_\Tot\) if the first duality map is an isomorphism for \(\Third=\Tot\).  Hence the existence of~\(\Theta\) is a harmless assumption for our purposes.

\begin{theorem}
  \label{the:second_duality}
  Fix \(\Tot\), \(\Dual\), \(D\) and~\(\Third\).  The second duality map \(\SPD^*\) is an isomorphism for all \(\Grd\)\nb-spaces~\(\Target\) if and only if there is \(\widetilde{\Theta}_\Third \in \Bic^n_{\Grd\ltimes\Tot}(\Third, \underline{\Dual}\times_\Base\Third)\) with the following properties:
  \begin{enumerate}[label=\textup{(\roman{*})}]
  \item \label{second_duality_1} \(\overline{\widetilde{\Theta}_\Third} \inpro_\Dual D = (-1)^n \Id_\Third\) in \(\Bic^0_\Grd(\Third, \Third)\);

  \item \label{second_duality_2} \((-1)^{in} \widetilde{\Theta}_\Third \inpro_\Third \overline{g} \inpro_\Dual D = g\) in \(\Bic^i_{\Grd\ltimes\Tot}(\Third,\Dual\times_\Base \Target)\) for all \(\Grd\)\nb-spaces~\(\Target\) and all \(g\in \Bic^i_{\Grd\ltimes\Tot}(\Third,\Dual\times_\Base \Target)\).
  \end{enumerate}
  Furthermore, \textup{\ref{second_duality_1}} determines~\(\widetilde{\Theta}_\Third\) uniquely, and the inverse of\/ \(\SPD^*\) is of the form
  \begin{equation}
    \label{eq:second_duality_inverse}
    \SPD\colon \Bic^{i-n}_\Grd(\Third,\Target) \to
    \Bic^i_{\Grd\ltimes\Tot}(\Third,\Dual\times_\Base\Target),
    \qquad
    f\mapsto \widetilde{\Theta}_\Third \inpro_\Third f.
  \end{equation}

  Suppose that \(\Theta\in\Bic^n_{\Grd\times\Tot}(\Tot, \underline{\Tot}\times_\Base\Dual)\) satisfies \(\Theta\inpro_\Dual D=\Id_\Tot\) in \(\Bic^n_{\Grd\times\Tot}(\Tot, \Tot)\).  Then the following conditions \ref{second_duality_3} and~\ref{second_duality_4} imply \ref{second_duality_1} and~\ref{second_duality_2}:
  \begin{enumerate}[resume, label=\textup{(\roman{*})}]
  \item \label{second_duality_3} the following diagram in~\(\Bic^*_\Grd\) commutes:
    \[
    \xymatrix{
      \Third
      \ar[r]^-{\overline{\widetilde{\Theta}_\Third}}
      \ar[d]_{\overline{\Third \times_\Tot \Theta}}&
      \Dual\times_\Base\Third
      \ar@{<->}[dl]_\cong^{(-1)^n\flip}\\
      \Third \times_\Base \Dual.
    }
    \]

  \item \label{second_duality_4} \(\widetilde{\Theta}_\Third \inpro_\Third \overline{g} = \Theta \inpro_\Tot g\) in \(\Bic^{i+n}_{\Grd\ltimes\Tot}(\Third, \underline{\Dual}\times_\Base \Dual\times_\Base\Target)\) for all \(g\in \Bic^i_{\Grd\ltimes\Tot}(\Third, \Dual\times_\Base \Target)\) and all \(\Grd\)\nb-spaces~\(\Target\).
  \end{enumerate}
\end{theorem}

\begin{proof}
  Condition~\ref{second_duality_1} means that \(\SPD^*(\widetilde{\Theta}_\Third) = (-1)^n\Id_\Third\) in \(\Bic^0_\Grd(\Third,\Third)\).  Hence there is a unique~\(\widetilde{\Theta}_\Third\) satisfying~\ref{second_duality_1} if \(\SPD^*\) is invertible.  Define a map~\(\SPD\) as in~\eqref{eq:second_duality_inverse}.  The defining property of~\(\widetilde{\Theta}_\Third\) and the graded commutativity of exterior products yield
  \[
  \SPD^*\circ \SPD(f)
  = (-1)^{in} \overline{\widetilde{\Theta}_\Third \inpro_\Third f}
  \inpro_\Dual D
  = (-1)^n \overline{\widetilde{\Theta}_\Third} \inpro_\Dual D
  \inpro_\Third f
  = f
  \]
  for all \(f\in\Bic^{i-n}_\Grd(\Third,\Target)\).  Hence the inverse of \(\SPD^*\) can only be~\(\SPD\).  We compute
  \[
  \SPD\circ\SPD^*(g)
  = (-1)^{in}\widetilde{\Theta}_\Third \inpro_\Third
  \overline{g} \inpro_\Dual D
  \]
  for all \(g\in \Bic^i_{\Grd\ltimes\Tot}(\Third,\Dual\times_\Base\Target)\), so that~\ref{second_duality_2} is equivalent to \(\SPD\circ\SPD^*=\Id\).

  Now assume that there is \(\Theta\in \Bic^n_{\Grd\ltimes\Tot}(\Tot,\underline{\Tot}\times_\Base\Dual)\) with \(\Theta\inpro_\Dual D= \Id_\Tot\).  Then Condition~\ref{second_duality_3} implies~\ref{second_duality_1} and~\ref{second_duality_4} implies~\ref{second_duality_2}, using the graded commutativity of exterior products.
\end{proof}

\begin{definition}
  \label{def:symmetric_dual}
  Let \(n\in\Z\) and let \(\Tot\) be a \(\Grd\)\nb-space.  A \emph{symmetric dual} for~\(\Tot\) is a quadruple \((\Dual,D,\Theta,\widetilde{\Theta})\), where
  \begin{itemize}
  \item \(\Dual\) is a \(\Grd\ltimes\Tot\)-space,

  \item \(D\in\Bic^{-n}_\Grd(\Dual,\Base)\),

  \item \(\Theta\in \Bic^n_{\Grd\ltimes\Tot} (\Tot, \underline{\Tot}\times_\Base\Dual) \cong \Coh^n_{\Grd,\Tot}(\underline{\Tot}\times_\Base\Dual)\) (see Definition~\ref{def:underline_factor}), and

  \item \(\widetilde{\Theta} \in \Bic^n_{\Grd\ltimes\Tot} (\Tot, \Tot\times_\Base \underline{\Dual}) \cong \Coh^n_{\Grd,\Tot}(\Tot\times_\Base \underline{\Dual})\)
  \end{itemize}
  satisfy the following conditions:
  \begin{enumerate}[label=\textup{(\roman{*})}]
  \item \label{sym_dual_i} \(\Theta \inpro_\Dual D = \Id_\Tot\) in the ring \(\Bic^0_{\Grd\ltimes\Tot}(\Tot,\Tot) \cong \Coh_\Grd^0(\Tot)\);

  \item \label{sym_dual_ii} \(\overline{(\Dual\times_\Tot\Theta)} \inpro_{\Dual\times_\Base\Dual} \flip = (-1)^n \overline{(\Dual\times_\Tot\Theta)}\) in \(\Bic^n_\Grd(\Dual,\Dual\times_\Base\Dual)\), where~\(\flip\) denotes the permutation \((x,y)\mapsto (y,x)\) on~\(\Dual\times_\Base\Dual\);

  \item \label{sym_dual_iii} \(\overline{\widetilde{\Theta}} = (-1)^n \overline{\Theta}\) in \(\Bic^n_\Grd(\Tot,\Tot\times_\Base\Dual)\);

  \item \label{sym_dual_iv} \(\Theta \inpro_\Dual \overline{(\Dual\times_\Tot g)} = \Theta \inpro_\Tot g\) in \(\Bic^{i+n}_{\Grd\ltimes\Tot}\bigl(\Tot, \underline{\Tot}\times_\Base \Dual\times_\Base\Target) \cong \Coh^{i+n}_{\Grd,\Tot}(\underline{\Tot}\times_\Base \Dual\times_\Base\Target)\) for all \(g\in \Bic^i_{\Grd\ltimes\Tot}(\Tot, \Tot\times_\Base \Target) \cong \Coh^i_{\Grd,\Tot}(\Tot\times_\Base\Target)\) and all \(\Grd\)\nb-spaces~\(\Target\);

  \item \label{sym_dual_v} \(\widetilde{\Theta} \inpro_\Tot \overline{g} = \Theta \inpro_\Tot g\) in \(\Bic^{i+n}_{\Grd\ltimes\Tot}\bigl(\Tot, \underline{\Dual}\times_\Base (\Dual\times_\Base\Target)\bigr) \cong \Coh^{i+n}_{\Grd,\Tot}(\underline{\Dual}\times_\Base (\Dual\times_\Base\Target)\bigr)\) for all \(g\in \Bic^i_{\Grd\ltimes\Tot}(\Tot, \Dual\times_\Base \Target) \cong \Coh^i_{\Grd,\Tot}(\Dual\times_\Base \Target)\) and all \(\Grd\)\nb-spaces~\(\Target\).
  \end{enumerate}
\end{definition}

We have used Theorem~\ref{the:geometric_K} repeatedly to simplify \(\Bic^*_{\Grd\ltimes\Tot}(\Tot,\blank)\) to \(\Coh^*_{\Grd,\Tot}(\blank)\).  Most of the data and conditions above take place in \(\Coh^*_{\Grd,\Tot}(\blank)\).

\begin{theorem}
  \label{the:duality_restricted}
  If the space~\(\Tot\) has a symmetric dual and if every \(\Grd\)\nb-equivariant vector bundle over~\(\Source\) is subtrivial, then the maps in \eqref{eq:first_duality}, \eqref{eq:first_duality_inverse}, \eqref{eq:second_duality}, and~\eqref{eq:second_duality_inverse} for \(\Third=\Tot\) yield isomorphisms
  \begin{align*}
    \Coh^i_{\Grd,\Tot}(\Tot\times_\Base\Target) \cong
    \Bic^i_{\Grd\ltimes\Tot}(\Tot,\Tot\times_\Base\Target)
    &\cong \Bic^{i-n}_\Grd(\Dual,\Target),\\
    \Coh^i_{\Grd,\Tot}(\Dual\times_\Base\Target) \cong
    \Bic^i_{\Grd\ltimes\Tot}(\Tot,\Dual\times_\Base\Target)
    &\cong \Bic^{i-n}_\Grd(\Tot,\Target)
  \end{align*}
  for all \(\Grd\)\nb-spaces~\(\Target\).
\end{theorem}

\begin{proof}
  The conditions for a symmetric dual in Definition~\ref{def:symmetric_dual} are \(\Theta\inpro_\Dual D= \Id_\Tot\) and the Conditions \ref{first_duality_3} and~\ref{first_duality_4} in Theorems \ref{the:first_duality} and~\ref{the:second_duality} with \(\Theta_\Tot=\Theta\) and \(\widetilde{\Theta}_\Tot=\widetilde{\Theta}\).  Hence the isomorphisms follow from Theorems \ref{the:geometric_K}, \ref{the:first_duality}, and~\ref{the:second_duality}.
\end{proof}

\begin{remark}
  \label{rem:duality_conditions_necessary}
  Theorem~\ref{the:duality_restricted} has a converse: the conditions in Definition~\ref{def:symmetric_dual} are necessary for the duality maps to be inverse to each other.  If~\(\Tot\) has a symmetric dual, then Conditions \ref{first_duality_3} and~\ref{first_duality_4} in Theorem~\ref{the:first_duality} are also necessary for the first duality isomorphism, and Conditions \ref{second_duality_3} and~\ref{second_duality_4} in Theorem~\ref{the:second_duality} are necessary for the second duality isomorphism.

  Analogous statements about duality isomorphisms in Kasparov theory are established in~\cite{Emerson-Meyer:Dualities}, and the proofs carry over almost literally.
\end{remark}

\begin{remark}
  \label{rem:duality_support}
  The variants of the duality isomorphisms with different support conditions considered in \cite{Emerson-Meyer:Dualities}*{Theorems 4.50 and 6.11} also work in our geometric theory, of course.  But we will not use these variants here.  We remark, however, that the constructions of symmetric duals below are sufficiently local to give duality isomorphisms with different support conditions as well.
\end{remark}

\subsection{Duality for certain \texorpdfstring{$\Grd$}{G}-spaces}
\label{sec:duality_smooth}

As before, \(\Grd\) is a numerably proper groupoid with object space~\(\Base\).  We are going to establish duality isomorphisms for spaces with certain properties.  The following definition lists our requirements:

\begin{definition}
  \label{def:tame}
  A \(\Grd\)\nb-space~\(\Tot\) is called \emph{normally non-singular} if there is a normally non-singular \(\Grd\)\nb-map from~\(\Tot\) to \(\Base\times[0,\infty)\) and if all \(\Grd\)\nb-vector bundles on~\(\Tot\) are subtrivial.
\end{definition}

Recall that a normally non-singular \(\Grd\)\nb-map from~\(\Tot\) to \(\Base\times[0,\infty)\) is a triple \(\NM\defeq (\VB,\Triv,\hat{f})\), where~\(\VB\) is a subtrivial \(\Grd\)\nb-vector bundle over~\(\Tot\), \(\Triv\) is a \(\Grd\)\nb-vector bundle over~\(\Base\), and~\(\hat{f}\) is an open embedding from~\(\total{\VB}\) into \(\total{\Triv}\times[0,\infty)\).

If there is a normally non-singular \(\Grd\)\nb-map from~\(\Tot\) to~\(\Base\), then there is one to \(\Base\times[0,\infty)\) as well because the map \(\Base\to\Base\times[0,\infty)\), \(\base\mapsto (\base,t)\), is the trace of a normally non-singular map for all \(t>0\).  Under some technical assumptions about equivariant vector bundles, normally non-singular maps \(\Tot\to\Base\) correspond to smooth structures on \(\Tot\times_\Base\Triv\) for some \(\Grd\)\nb-vector bundle~\(\Triv\) over~\(\Base\), and normally non-singular maps \(\Tot\to\Base\times[0,\infty)\) correspond to a structure of smooth \(\Grd\)\nb-manifold with boundary on \(\Tot\times_\Base\Triv\) for some \(\Grd\)\nb-vector bundle~\(\Triv\) over~\(\Base\).  The technical assumptions here are related to the finite orbit type assumption in the Mostow Embedding Theorem.  For instance, if~\(\Grd\) is a compact group, then a smooth \(\Grd\)\nb-manifold with boundary~\(\Tot\) admits a normal map to~\(\Base\times[0,\infty)\) if and only if it has finite orbit type.  But this assumption does not yet ensure that all \(\Grd\)\nb-vector bundles over~\(\Tot\) are subtrivial. For example, let \(\Tot\) be the integers and~\(\Grd\) be the circle with the trivial action on~\(\Tot\).  Using the identification \(\Tot \cong \widehat{\Grd}\) we get an obvious \(\Grd\)\nb-equivariant complex line bundle on~\(\Tot\) which is not subtrivial because it contains infinitely many inequivalentirreducible representations of~\(\Grd\).  For more information on non-singular spaces, see~\cite{Emerson-Meyer:Normal_maps}.

Let~\(\Tot\) be a no-singular \(\Grd\)\nb-space and let \(\NM\defeq (\VB,\tilde{\Triv},\hat{f})\) be a normally non-singular \(\Grd\)\nb-map from~\(\Tot\) to \(\Base\times[0,\infty)\).  \emph{We assume~\(\tilde{\Triv}\) to be \(\Coh\)\nb-oriented and with a well-defined dimension, and we let \(d\defeq \dim \tilde{\Triv}+1\).}  We impose no restrictions on the \(\Grd\)\nb-vector bundle~\(\VB\); thus the normally non-singular map~\(\NM\) need not be \(\Coh\)\nb-oriented.

\begin{remark}
  \label{rem:assumptions_automatic}
  Assume that any \(\Grd\)\nb-vector bundle over~\(\Base\) is a direct summand in an \(\Coh\)\nb-oriented one; this is automatic if~\(\Coh\) is cohomology, equivariant \(\K\)\nb-theory, or equivariant \(\KO\)\nb-theory.  Then a lifting of our original normally non-singular map replaces~\(\tilde{\Triv}\) by an \(\Coh\)\nb-oriented \(\Grd\)\nb-vector bundle.  Hence our assumption that~\(\tilde{\Triv}\) be \(\Coh\)\nb-oriented is no loss of generality.

  Similarly, if the fibre dimensions of~\(\Triv\) are merely bounded above by some \(N\in\N\), then lifting along the locally constant \(\Grd\)\nb-vector bundle with fibre \(\R^{N-\dim\Triv_\base}\) at~\(\base\) ensures that \(\dim \tilde{\Triv}_\base=N\) for all \(\base\in\Base\), without affecting the \(\Coh\)\nb-orientation.  Hence our assumptions on~\(\tilde{\Triv}\) can always be achieved by lifting~\(\NM\).
\end{remark}

We use~\(\hat{f}\) to identify~\(\total{\VB}\) with an open subset of \(\total{\tilde{\Triv}}\times[0,\infty)\) and thus drop~\(\hat{f}\) from our notation from now on.  Let \(\bd\VB\defeq \total{\VB}\cap \total{\tilde{\Triv}}\times\{0\}\).  This is an open subset of \(\total{\tilde{\Triv}}\times\{0\} \cong \total{\tilde{\Triv}}\).  Let
\[
\Triv\defeq \tilde{\Triv}\oplus\R
\qquad\text{and}\qquad
\Dual \defeq
\bigl(\bd\VB\times(-\infty,0]\bigr) \cup\total{\VB}.
\]
Then~\(\Dual\) is an open subset of~\(\total{\Triv}\).  Since~\(\total{\Triv}\) is \(\Coh\)\nb-oriented and \(d\)\nb-dimensional, the special normally non-singular submersion \(\Dual\opem\total{\Triv}\epi\Base\) provides \(D\in\Bic^{-d}_\Grd(\Dual,\Base)\).

There is a canonical deformation retraction from~\(\Dual\) onto
\(\total{\VB}\subseteq\Dual\):
\[
h\colon \Dual\times[0,1]\to\Dual,\qquad
h\bigl((\vb,s),t\bigr) \defeq (\vb,s\cdot t)
\quad\text{for \(\vb\in\bd\VB\), \(s\in(-\infty,0]\),
  \(t\in[0,1]\),}
\]
and \(h(\vb,t)=\vb\) for \(\vb\in\total{\VB}\), \(t\in[0,1]\).  We view~\(\Dual\) as a space over~\(\Tot\) using the map \(\proj{\VB}\circ h_0\colon \Dual\to\total{\VB}\epi\Tot\).

This construction simplifies if we use a normally non-singular \(\Grd\)\nb-map \((\VB,\Triv,\hat{f})\) from~\(\Tot\) to~\(\Base\).  Then \(\Dual\defeq \total{\VB}\), viewed as a space over~\(\Tot\) via \(\proj{\VB}\colon \total{\VB}\epi\Tot\); the pair \((\Triv,\hat{f})\) is a special normally non-singular submersion from~\(\Dual\) to~\(\Base\) and provides \(D\in\Bic^{-d}_\Grd(\Dual,\Base)\).

Specialising further, if~\(\Grd\) is a compact group, so that \(\Base=\pt\), then a normally non-singular map \(\Tot\to\Base\) is equivalent to a tubular neighbourhood \(\hat{f}\colon \total{\VB} \opem \R^n\) for an embedding \(\Tot\to\R^n\) with normal bundle~\(\VB\).

We have now described the ingredients \(\Dual\) and~\(D\) of the symmetric dual, which fix the duality isomorphisms \(\PD^*\) and \(\SPD^*\) by \eqref{eq:first_duality} and~\eqref{eq:second_duality} and thus determine the other ingredients \(\Theta\) and~\(\widetilde{\Theta}\).  We now describe these.

Since~\(\Dual\) is an open subset of~\(\total{\Triv}\), \(\Tot\times_\Base \Dual\) is an open subset of \(\Tot\times_\Base \total{\Triv}\).  The latter is the total space of the trivial \(\Grd\)\nb-vector bundle~\(\Triv^\Tot\) over~\(\Tot\).  The zero section of~\(\VB\) followed by the embedding \(\total{\VB} \defeq \Dual\opem \total{\Triv}\) provides a \(\Grd\)\nb-equivariant section of~\(\Triv^\Tot\).

\begin{lemma}
  \label{lem:sections_normal}
  Any section of~\(\Triv^\Tot\) with values in \(\Tot\times_\Base \Dual \subset \Tot\times_\Base \Triv = \Triv^\Tot\) is the trace of a \(\Grd\ltimes\Tot\)\nb-equivariant normally non-singular embedding from~\(\Tot\) to \(\underline{\Tot}\times_\Base\Dual\).
\end{lemma}

\begin{proof}
  This is an easy special case of the equivariant Tubular Neighbourhood Theorem, which we establish by hand.  We want to define a \(\Grd\ltimes\Tot\)\nb-equivariant open embedding \(\iota\colon \total{\Triv^\Tot} = \Tot\times_\Base \total{\Triv} \to \underline{\Tot}\times_\Base\Dual\) by
  \[
  \iota(\tot,\triv) \defeq \bigl(\tot, \zers{\VB}(\tot) + \triv\cdot R(\zers{\VB}(\tot),\triv)\bigr)
  \qquad \text{with}\quad
  R(\dual,\triv) = \frac{\varrho(\dual)}{\norm{\triv}+1};
  \]
  here \(\zers{\VB}\colon \Tot \to \total{\VB} \subseteq \total{\Triv}\) is the zero section of~\(\VB\) and \(+\) and~\(\cdot\) denote the addition and scalar multiplication in the vector bundle~\(\Triv\); \(\norm{e}\) is the norm from a \(\Grd\)\nb-invariant inner product on~\(\Triv\), which exists by \cite{Emerson-Meyer:Normal_maps} because~\(\Grd\) is numerably proper; and \(\varrho\colon \Dual\to(0,1]\) is a \(\Grd\)\nb-invariant function chosen such that \(\dual + \triv\in \Dual\) for \((\dual,\triv)\in \Dual\times_\Base \total{\Triv}\) with \(\norm{\triv}<\varrho(\dual)\) -- this ensures that \(\iota(\tot,\triv) \in \Tot\times_\Base\Dual\) for all \((\tot,\triv) \in \Tot\times_\Base \total{\Triv}\).

  For each \(\dual\in\Dual\), there are an open neighbourhood~\(U_\dual\) and \(\epsilon_\dual>0\) such that \(\dual' + \triv\in \Dual\) for \((\dual',\triv)\in \Dual\times_\Base \total{\Triv}\) with \(\dual'\in U_\dual\) and \(\norm{\triv}<\epsilon_\dual\) because~\(\Dual\) is open in~\(\total{\Triv^\Tot}\).  Since~\(\Dual\) is paracompact by our standing assumption on topological spaces, we may use a partition of unity to find a continuous function \(\varrho\colon \Dual\to(0,\infty)\) with \(\varrho(\dual)\le\epsilon_\dual\) for all \(\dual\in\Dual\).  We can replace~\(\varrho\) by a \(\Grd\)\nb-invariant function, using that~\(\Grd\) is numerably proper.
\end{proof}

Since the \(\Grd\)\nb-vector bundle~\(\Triv\) is \(\Coh\)\nb-oriented, so is the normally non-singular embedding \((\Triv^\Tot,\iota)\) from~\(\Tot\) to \(\Tot\times_\Base\Dual\).  We let
\[
\Theta\defeq (\Triv^\Tot,\iota)!\in
\Bic^d_{\Grd\ltimes\Tot}(\Tot, \underline{\Tot}\times_\Base \Dual)
\cong \Coh^d_{\Grd,\Tot}(\underline{\Tot}\times_\Base\Dual).
\]

Now we modify this construction to get~\(\widetilde{\Theta}\).  View~\(\total{\Triv^\Tot}\) as a space over~\(\Tot\) via
\[
\pi'\colon \total{\Triv^\Tot}\to\Tot,\qquad
(\tot,\triv)\mapsto \proj{\VB}\circ h_0\bigl(\zers{\VB}(\tot) +
e\cdot R(\zers{\VB}(\tot),\triv)\bigr).
\]
Let \(\Disk(\Triv^\Tot)\) be the unit disk bundle of~\(\Triv^\Tot\) with respect to the chosen metric.  The restrictions of both \(\proj{\Triv^\Tot}\) and~\(\pi'\) to \(\Disk(\Triv^\Tot)\) are proper maps to~\(\Tot\), that is, \(\Disk(\Triv^\Tot)\) is \(\Tot\)\nb-compact when we view~\(\total{\Triv^\Tot}\) as a space over~\(\Tot\) using one of these maps.  Even more, \(\Disk(\Triv^\Tot)\times[0,1]\) is \(\Tot\)\nb-compact with respect to
\[
\bar\pi\colon \total{\Triv^\Tot}\times[0,1] \to
\Tot,\qquad
(\tot,\triv,t) \mapsto \pi'(\tot,t\cdot\triv);
\]
this is a homotopy between \(\proj{\Triv^\Tot}\) and~\(\pi'\).

The \(\Coh\)\nb-orientation~\(\Thom_\Triv\) of~\(\Triv^\Tot\) may be represented by a cohomology class supported in \(\Disk(\Triv^\Tot)\).  Hence we get \(\bar\Thom_\Triv\in \Coh^*_{\Grd,\Tot} \bigl(\total{\Triv^\Tot}\times[0,1],\bar\pi\bigr)\) and, by restriction, \(\Thom_\Triv\in \Coh^*_{\Grd,\Tot}\bigl(\total{\Triv^\Tot},\pi'\bigr)\); here the maps on the right specify how to view the spaces on the left as spaces over~\(\Tot\).

Now we can describe \(\widetilde{\Theta} \in \Bic^d_{\Grd\ltimes\Tot}(\Tot, \Tot\times_\Base \underline{\Dual})\): it is the class of the special correspondence \((0,\total{\Triv^\Tot},\pi',\Thom_\Triv)\) from~\(\Tot\) to \(\Tot\times_\Base \underline{\Dual}\).  Here we view~\(\total{\Triv^\Tot}\) as an open subset of \(\Tot\times_\Base \underline{\Dual}\) using the map~\(\iota\) constructed above and as a space over~\(\Tot\) using~\(\pi'\).  As a class in \(\Coh^d_{\Grd,\Tot}(\Tot\times_\Base \underline{\Dual})\), we have \(\widetilde{\Theta}=\iota!(\Thom_\Triv)\).

\begin{theorem}
  \label{the:duality_smooth}
  Let~\(\Tot\) be a normally non-singular \(\Grd\)\nb-space.  Assume that any \(\Grd\)\nb-vector bundle over~\(\Base\) is contained in an \(\Coh\)\nb-oriented one.  The data \((\Dual,D,\Theta,\widetilde{\Theta})\) provides a symmetric dual for~\(\Tot\).  Hence there are duality isomorphisms
  \begin{align*}
    \Coh^i_{\Grd,\Tot}(\Tot\times_\Base\Target)
    &\cong \Bic^{i-n}_\Grd(\Dual,\Target),\\
    \Coh^i_{\Grd,\Tot}(\Dual\times_\Base\Target)
    &\cong \Bic^{i-n}_\Grd(\Tot,\Target)
  \end{align*}
  for all \(\Grd\)\nb-spaces~\(\Target\).
\end{theorem}

\begin{proof}
  We must check Conditions \ref{sym_dual_i}--\ref{sym_dual_v} in Definition~\ref{def:symmetric_dual}.  The duality isomorphisms then follow from Theorem~\ref{the:duality_restricted}.

  The composition of~\(\iota\) with the open embedding \(\Dual\opem\total{\Triv}\) is isotopic to the identity map on \(\total{\Triv^\Tot}\).  Thus the computation in \cite{Emerson-Meyer:Normal_maps}*{Example 4.25} yields Condition~\ref{sym_dual_i}:
  \[
  \Theta\inpro_\Dual D = \Id_\Tot
  \qquad \text{in \(\Bic^0_{\Grd\ltimes\Tot}(\Tot,\Tot)\).}
  \]

  Condition~\ref{sym_dual_ii} asserts that \(\Dual\times_\Tot \Theta\in \Bic^d_\Grd(\Dual,\Dual\times_\Base \Dual)\) is rotation invariant up to the sign \((-1)^d\).  This amounts to an assertion about the map~\(\Dual\times_\Tot\iota\).  First we construct an isotopy from \(\Dual\times_\Tot\iota\) to a slightly more symmetric map.  By definition,
  \[
  \Dual\times_\Tot\iota\colon
  \Dual\times_\Tot \Tot\times_\Base \total{\Triv}
  \cong \Dual\times_\Base \total{\Triv}
  \opem \Dual\times_\Base \Dual
  \]
  maps
  \[
  (\dual,\triv) \mapsto
  \bigl(\dual, 0\cdot h_0(\dual)+ \triv\cdot R(0\cdot h_0(\dual),\triv)\bigr),
  \]
  where \(h_0\colon \Dual \to \total{\VB}\) is the canonical retraction and~\(\cdot\) is the scalar multiplication in~\(\VB\), that is, \(0\cdot\vb = \zers{\VB}\circ\proj{\VB}(\vb)\).  The open embeddings
  \[
  (\dual,\triv) \mapsto
  \bigl(\dual, t\cdot h_t(\dual)+ \triv\cdot R(t\cdot h_t(\dual),\triv)\bigr)
  \]
  for \(t\in[0,1]\) provide an isotopy from \(\Dual\times_\Tot\iota\) to the map
  \[
  \iota_\Dual\colon \Dual\times_\Base \total{\Triv}
  \opem \Dual\times_\Base \Dual,\qquad
  (\dual,\triv) \mapsto
  \bigl(\dual, \dual + \triv\cdot R(\dual,\triv)\bigr).
  \]

  The matrices
  \[
  A_t \defeq \begin{pmatrix}1+t&-t\\t&1-t\end{pmatrix},\qquad
  A_t^{-1} \defeq \begin{pmatrix}1-t&t\\-t&1+t\end{pmatrix}
  \]
  are inverse to each other.  Applying~\(A_t\) for \(t\in[0,1]\) to~\(\iota_\Dual\) provides an isotopy
  \[
  (\dual,\triv) \mapsto
  A_t\cdot \begin{pmatrix}
    \dual\\\dual+\triv\cdot R(\dual,\triv)
  \end{pmatrix}
  = \begin{pmatrix}
    \dual-\triv\cdot R(\dual,\triv) t\\
    \dual+\triv\cdot R(\dual,\triv)\cdot (1-t)
  \end{pmatrix}
  \]
  of open embeddings \(\total{\Triv^\Dual}\opem \total{\Triv}\times_\Base\total{\Triv}\) between~\(\iota_\Dual\) and \(\flip\circ\iota_\Dual\circ\Psi\), where \(\Psi(\dual,\triv) \defeq (\dual,-\triv)\).  These maps take values in \(\Dual\times_\Base\Dual\) by construction of \(R(\dual,\triv)\).  The isotopies \(\Dual\times_\Base \iota \sim \iota_\Dual \sim \flip\circ\iota_\Dual\circ\Psi \sim \flip\circ(\Dual\times_\Base\iota)\circ\Psi\) constructed above provide an open embedding
  \begin{equation}
    \label{eq:opem_isotopy_ii_and_v}
    \kappa\colon \total{\Triv^\Dual}\times[0,1]
    \opem\Dual\times_\Base\Dual\times[0,1]
  \end{equation}
  with \(\kappa_0 =\Dual\times_\Tot\iota\) and \(\kappa_1 = \flip \circ (\Dual\times_\Tot\iota)\circ \Psi\).  The linear isomorphism~\(\Psi\) has the class \((-1)^{\dim \Triv}=(-1)^d\) in \(\Bic^0_{\Grd\ltimes\Tot}(\Tot\times_\Base \Triv, \Tot\times_\Base\Triv)\).  The isotopy of open embeddings~\(\kappa\) together with the \(\Coh\)\nb-orientation \(\tau_\Triv\in \Coh^d_{\Grd,\Dual}(\total{\Triv^\Dual})\) provides a special bordism between \(\Dual\times_\Tot\Theta\) and \((-1)^d (\Dual\times_\Tot\Theta) \inpro \flip\).  This finishes the proof of Condition~\ref{sym_dual_ii}.

  Condition~\ref{sym_dual_iii} asserts that the forgetful functor to~\(\Bic_\Grd\) maps \(\Theta\) and \((-1)^d\widetilde{\Theta}\) to the same element in \(\Bic^d_\Grd(\Tot, \Tot\times_\Base\Dual)\).  The homotopy~\(\bar\pi\) between the projection \(\bar\pi_1=\pi'\colon \total{\Triv^\Tot}\to\Tot\) and the standard projection \(\bar\pi_0=\proj{\Triv^\Tot}\) and~\(\bar\Thom_\Triv\) provide a special bordism between \((-1)^d\widetilde{\Theta} \defeq (\total{\Triv^\Tot},0,\pi',\Thom_\Triv)\) and \((\total{\Triv^\Tot},0,\proj{\Triv^\Tot},\Thom_\Triv)\).  The latter is the special correspondence associated to~\(\Theta\).  This proves Condition~\ref{sym_dual_iii}.

  Condition~\ref{sym_dual_iv} asserts \(\Theta \inpro_\Dual \overline{(\Dual\times_\Tot g)} = \Theta \inpro_\Tot g\) for \(g\in \Bic^i_{\Grd\ltimes\Tot}(\Tot, \Tot\times_\Base \Target)\) for any \(\Grd\)\nb-space~\(\Target\).  We may pull back~\(g\) to a class in \(\Bic^i_{\Grd\ltimes (\Tot\times_\Base\Dual)}(\Tot\times_\Base\Dual, \Tot\times_\Base\Dual\times_\Base\Target)\) along the coordinate projection \(\Tot\times_\Base\Dual\to\Tot\) and along the projection \(\Tot\times_\Base\Dual\to\Tot\), \((\tot,\dual)\mapsto 0\cdot h_0(\dual)\).  The products \(\Theta \inpro_\Tot g\) and \(\Theta \inpro_\Dual \overline{(\Dual\times_\Tot g)}\) are the composition products of~\(\Theta\) with these two pull-backs of~\(g\).  These composites agree because the two maps \(\Tot\times_\Base\Dual\rightrightarrows\Tot\) restrict to homotopic maps on the range of the embedding~\(\iota\) (use~\(\bar\pi\)).  The pull-back of~\(g\) along this homotopy provides a bordism of correspondences that connects the two products in question.  This finishes the proof of Condition~\ref{sym_dual_iv}.

  Condition~\ref{sym_dual_v} asserts \(\widetilde{\Theta} \inpro_\Tot \overline{\Kclass} = \Theta \inpro_\Tot \Kclass\) for \(\Kclass\in \Bic^i_{\Grd\ltimes\Tot}(\Tot, \Dual\times_\Base \Target) \cong \Coh^i_{\Grd,\Tot}(\Dual\times_\Base\Target)\) for any \(\Grd\)\nb-space~\(\Target\).  We may rewrite these products as composition products:
  \begin{equation}
    \label{eq:composition_v}
    \widetilde{\Theta} \inpro_\Tot \overline{\Kclass}
    = \widetilde{\Theta} \inpro_{\Tot\times_\Base\underline{\Dual}}
    (\Kclass\times_\Base \Dual),\qquad
    \Theta \inpro_\Tot \Kclass =
    \Theta \inpro_{\underline{\Tot}\times_\Base\Dual}
    (\Kclass\times_\Base\Dual)
  \end{equation}
  with \(\Kclass\times_\Base\Dual \in \Coh^i_{\Grd,\Tot\times_\Base\Dual} (\Dual\times_\Base\Target\times_\Base\Dual) \cong \Bic^i_{\Grd\ltimes (\Tot\times_\Base\Dual)} (\Tot\times_\Base\Dual, \Dual\times_\Base\Target\times_\Base\Dual)\).  The products in~\eqref{eq:composition_v} lie in \(\Bic^{i+n}_{\Grd\ltimes\Tot}(\Tot, \Dual\times_\Base\Target\times_\Base\underline{\Dual})\) and \(\Bic^{i+n}_{\Grd\ltimes\Tot}(\Tot, \underline{\Dual}\times_\Base\Target\times_\Base\Dual)\), respectively, so that we must, more precisely, show that
  \[
  \widetilde{\Theta} \inpro (\Kclass\times_\Base \Dual) =
  \Theta \inpro (\Kclass\times_\Base\Dual) \inpro \flip.
  \]

  The isomorphism in Theorem~\ref{the:geometric_K} replaces \(\Kclass\times_\Base\Dual\in \Coh^i_{\Grd,\Tot\times_\Base\Dual} (\Dual\times_\Base\Target\times_\Base\Dual)\) by the special correspondence \((0,\Dual\times_\Base\Target\times_\Base\Dual,\mapl, \Kclass\times_\Base\Dual)\) with
  \[
  \mapl\colon \Dual\times_\Base\Target\times_\Base\Dual \to
  \Tot\times_\Base\Dual,\qquad
  (\dual_1,\target,\dual_2)\mapsto (\proj{\VB} h_0(\dual_1),\dual_2).
  \]
  Recall that \(\Theta\) and~\(\widetilde{\Theta}\) are represented by the \(\Grd\ltimes\nobreak\Tot\)-equivariant special correspondences \((\total{\Triv^\Tot},0,\proj{\Triv^\Tot},\Thom_\Triv)\) and \((\total{\Triv^\Tot},0,\pi',(-1)^d\Thom_\Triv)\), where we use~\(\iota\) to view \(\total{\Triv^\Tot}\) as an open subset of~\(\Tot\times_\Base\Dual\).

  Having represented our bivariant cohomology classes by special correspondences, we may use the definition to compute the composition products in~\eqref{eq:composition_v}.  In both cases, the \(\Grd\)\nb-space in the middle is \(\Midd\defeq \total{\Triv}\times_\Base\Dual\times_\Base \Target\), viewed as an open subset of \(\Dual\times_\Base\Target\times_\Base\Dual\) via
  \begin{multline*}
    \lambda\colon
    \Midd\defeq \total{\Triv}\times_\Base\Dual\times_\Base
    \Target \opem
    \Dual\times_\Base\Target\times_\Base\Dual,\\
    (\triv,\dual,\target)\mapsto \bigl(\dual,\target,
    0\cdot h_0(\dual)+ e\cdot R(0\cdot h_0(\dual),e)\bigr),
  \end{multline*}
  and the \(\Coh\)\nb-class on~\(\Midd\) is \(\Thom_\Triv\cdot\lambda^*(\Kclass)\).  But the maps to~\(\Tot\) are different: for \(\Theta \inpro_\Tot \Kclass\), we use \(\underline{\Dual}\times_\Base\Target\times_\Base\Dual\) and thus view~\(\Midd\) as a space over~\(\Tot\) via \((\triv,\dual,\target)\mapsto \proj{\VB}h_0(\dual)\); for \(\widetilde{\Theta} \inpro_\Tot \overline{\Kclass}\), we use \(\Dual\times_\Base\Target\times_\Base\underline{\Dual}\) and thus view~\(\Midd\) as a space over~\(\Tot\) via \((\triv,\dual,\target)\mapsto \proj{\VB}h_0\bigl( 0\cdot h_0(\dual) + \triv\cdot R(0\cdot h_0(\dual),\triv)\bigr)\).  Thus, we must compose one of the copies of~\(\lambda\) with the flip isomorphism
  \[
  \flip\colon
  \underline{\Dual}\times_\Base\Target\times_\Base\Dual
  \xrightarrow{\cong}
  \Dual\times_\Base\Target\times_\Base\underline{\Dual},\qquad
  (\dual_1,\target,\dual_2)\mapsto
  (\dual_2,\target,\dual_1)
  \]
  before we can compare them.

  Now we use the isotopy of open embeddings~\(\kappa\) in~\eqref{eq:opem_isotopy_ii_and_v}.  It connects \(\lambda=\kappa_0\times_\Base\Target\) and \(\flip\circ\lambda\circ\Psi=\kappa_1\times_\Base\Target\), where~\(\Psi\) maps \(\triv\mapsto-\triv\) on~\(\Triv\).  The \(\Coh\)\nb-cohomology class \(\tau_\Triv\cdot \kappa^*(\Kclass)\) on \(\total{\Triv}\times_\Base\Dual\times_\Base\Target\times[0,1]\) has \(\Tot\)\nb-compact support with respect to the map \(\underline{\Dual}\times_\Base\Target\times_\Base\Dual\to\Tot\).  This produces a special bordism between the products \(\Theta \inpro_\Tot \Kclass\) and \(\widetilde{\Theta} \inpro_\Tot \overline{\Kclass}\).  Notice that the sign~\((-1)^d\) in the definition of~\(\widetilde{\Theta}\) cancels the sign produced by the automorphism~\(\Psi\) above.
\end{proof}

\begin{remark}
  \label{rem:duality_locally_trivial}
  With some additional effort, it can be shown more generally that the maps \(\PD^*\) and \(\SPD^*\) provide isomorphisms
  \begin{align*}
    \Bic_{\Grd\ltimes\Tot}^i
    (\Third,\Dual\times_\Base\Target) &\cong
    \Bic_\Grd^{i-n}(\Third,\Target),\\
    \Bic^i_{\Grd\ltimes\Tot}(\Third,
    \Tot\times_\Base\Target)&\cong
    \Bic^{i-n}_\Grd(\Dual\times_\Tot\Third, \Target),
  \end{align*}
  provided~\(\Tot\) is a non-singular \(\Grd\)\nb-space with boundary, \(\Target\) is any \(\Grd\)\nb-space, and~\(\Third\) is a \emph{locally trivial} \(\Grd\ltimes\Tot\)\nb-space.  Local triviality means that there is a neighbourhood~\(W\) of the diagonal in~\(\Tot\times_\Base\Tot\) such that the pull-backs of~\(\Third\) to~\(W\) along the two coordinate projections \(W\to\Tot\) become \(\Grd\)\nb-equivariantly homeomorphic.  The same condition is used in~\cite{Emerson-Meyer:Dualities} to construct duality isomorphisms in equivariant Kasparov theory.
\end{remark}

\section{Comparison to Kasparov theory}
\label{sec:compare_Kasparov}

Now we restrict attention to the case where our cohomology theory is equivariant \(\K\)\nb-theory or, more precisely, the representable equivariant \(\K\)\nb-theory for locally compact \(\Grd\)\nb-spaces for a proper locally compact groupoid~\(\Grd\), see~\cite{Emerson-Meyer:Equivariant_K}.  We denote the topological bivariant \(\K\)\nb-theory defined above by
\[
\GKK^*_\Grd(\Source,\Target)
\defeq \Bic^*_\Grd(\Source,\Target)
\qquad\text{for \(\Coh=\RK\)}.
\]
We want to compare it to the equivariant Kasparov theory defined in~\cite{LeGall:KK_groupoid}.  In order for both theories to be defined, we require \(\Source\) and~\(\Target\) to be second countable, locally compact Hausdorff spaces and~\(\Grd\) to be a proper, second countable, locally compact, Hausdorff groupoid with Haar system.  Then~\(\Grd\) is numerably proper by \cite{Emerson-Meyer:Normal_maps}*{Lemma 2.16}.

\begin{remark}
  \label{rem:vb_K}
  It follows from~\cite{Emerson-Meyer:Equivariant_K} that if~\(\Grd\) is a proper groupoid, \(\Source\) is a cocompact \(\Grd\)\nb-space, and~\(\Target\) is a \(\Grd\)\nb-space with enough \(\Grd\)\nb-vector bundles, then \(\RK^*_{\Grd, \Source}(\Target)\) can be decribed in terms of triples \((\VB_+, \VB_-, \varphi)\), where~\(\VB_\pm\) are \(\Grd\)\nb-vector bundles on~\(\Target\) and \(\varphi\colon \VB_+ \to \VB_-\) is an equivariant vector bundle map that is an isomorphism off an \(\Source\)\nb-compact \(\Grd\)\nb-invariant closed subset.  If~\(\Target\) is also a smooth \(\Grd\)\nb-manifold, then a simple argument with crossed products implies that the vector bundles and~\(\varphi\) can be taken to be smooth.  Combining these observations with Theorem~\ref{the:smooth_duality} gives a description of \(\KK^\Grd_*(\CONT_0(\Source), \CONT_0(\Target))\) in terms of smooth correspondences whose \(\K\)\nb-theory data are encoded by smooth \(\Grd\)\nb-equivariant vector bundles on~\(\Target\) which are smoothly isomorphic off an \(\Source\)\nb-compact set.  This is more in line with the traditional definitions.
\end{remark}

\begin{theorem}
  \label{the:compare_to_KK}
  Let~\(\Grd\) be a proper, second countable, locally compact groupoid with Haar system and let \(\Source\) and~\(\Target\) be second countable, locally compact \(\Grd\)\nb-spaces.  There is a natural transformation \(\GKK^*_\Grd(\Source,\Target) \to \KK^\Grd_* \bigl(\CONT_0(\Source), \CONT_0(\Target)\bigr)\) that preserves gradings, composition products, and exterior products.  It is an isomorphism if~\(\Source\) has a symmetric dual in \(\GKK^*\).
\end{theorem}

\begin{proof}
  Let \(\Cor \defeq (\Triv,\Midd,\mapl,\Kclass)\) be a special correspondence from~\(\Source\) to~\(\Target\).  The equivariant \(\K\)\nb-theory \(\RK^*_{\Grd,\Source}(\Midd)\) of~\(\Midd\) with \(\Source\)\nb-compact support is identified in~\cite{Emerson-Meyer:Equivariant_K} with \(\KK_*^{\Grd\ltimes\Source}\bigl(\CONT_0(\Source),\CONT_0(\Midd)\bigr)\); here we view~\(\Midd\) as a space over~\(\Source\) using~\(\mapl\).  In particular, \(\Kclass\) becomes a class in \(\KK_*^{\Grd\ltimes\Source} \bigl(\CONT_0(\Source), \CONT_0(\Midd)\bigr)\), which maps to \(\KK_*^\Grd \bigl(\CONT_0(\Source), \CONT_0(\Midd)\bigr)\) by a forgetful functor.  Since~\(\Midd\) is an open subset of~\(\total{\Triv^\Target}\), we may identify \(\CONT_0(\Midd)\) with an ideal in \(\CONT_0(\total{\Triv^\Target})\).  The \(\K\)\nb-orientation for the \(\Grd\)\nb-vector bundle~\(\Triv^\Target\) over~\(\Target\) induces a \(\KK^{\Grd\ltimes\Target}\)\nb-equivalence between \(\CONT_0(\total{\Triv^\Target})\) and \(\CONT_0(\Target)\).  Putting all ingredients together, we get a class in \(\KK^\Grd_0 \bigl(\CONT_0(\Source), \CONT_0(\Target)\bigr)\), which we denote by \(\KK(\Cor)\).

  The invertible element in \(\KK^{\Grd\ltimes\Target} \bigl(\CONT_0(\total{\Triv^\Target}), \CONT_0(\Target)\bigr)\) used above induces the Thom isomorphism for the \(\K\)\nb-oriented \(\Grd\)\nb-vector bundle~\(\Triv^\Target\).  Since the composition of two Thom isomorphisms is again a Thom isomorphism for the direct sum vector bundle, it follows that \(\KK(\Cor) = \KK(\Cor^{\Triv'})\) if~\(\Cor^{\Triv'}\) is the Thom modification of~\(\Cor\) along a \(\K\)\nb-oriented \(\Grd\)\nb-vector bundle~\(\Triv'\) over~\(\Base\).  Recall also that a special bordism of correspondences from~\(\Source\) to~\(\Target\) is nothing but a special correspondence from~\(\Source\) to~\(\Target\times[0,1]\).  Hence a bordism produces a homotopy of Kasparov cycles, so that \(\KK(\Cor)\) is invariant under special bordisms.  Thus \(\Cor\mapsto \KK(\Cor)\) is a well-defined map \(\GKK_\Grd^*(\Source,\Target) \to \KK^\Grd_* \bigl(\CONT_0(\Source), \CONT_0(\Target)\bigr)\).

  It is clear that this construction preserves gradings.  Compatibility with exterior products is easy to check as well.  We check compatibility with composition products.  Recall that a special correspondence \(\Cor = (\Triv,\Midd,\mapl,\Kclass)\) is the product of one of the form \((\mapl,\Kclass)^*\) and~\(\mapr!\) for a special normally non-singular submersion \(\mapr\colon \Midd\subseteq \total{\Triv^\Target}\epi\Target\).  By construction, \(\KK(\Cor)\) is the product of \(\KK\bigl((\mapl,\Kclass)^*\bigr)\) and \(\KK(\mapr!)\).  It is easy to see that \(\Cor\mapsto\KK(\Cor)\) is multiplicative on correspondences of the special form \((\mapl,\Kclass)^*\).  Multiplicativity for normally non-singular maps -- in particular, for special normally non-singular submersions -- follows as in \cite{Emerson-Meyer:Normal_maps}.

  It remains to check multiplicativity for products of the form \(\NM!  \inpro (\mapl,\Kclass)^*\) for a special normally non-singular submersion \(\NM= (\Triv,\Source)\) with an open subset \(\Source \subseteq \total{\Triv^\Target}\), a \(\Grd\)\nb-map \(\mapl\colon \Third\to\Target\), and \(\Kclass\in\RK^*_{\Grd,\Target}(\Third)\).  The product in \(\GKK^*_\Grd(\Source,\Third)\) is the special correspondence
  \[
  \bigl(\Triv,\Source\times_\Target\Third,p_1,p_2^*(\Kclass)\bigr),
  \]
  where \(p_1\colon \Source\times_\Target\Third\to\Source\) and \(p_2\colon \Source\times_\Target\Third\to\Third\) are the coordinate projections.  Now the assertion follows from the naturality of the Thom isomorphism with respect to \(\KK\bigl((\mapl,\Kclass)^*\bigr)\).  More explicitly, let \(\proj{\Triv^\Target}\colon \total{\Triv^\Target}\epi\Target\) be the projection.  Then the diagram
  \[
  \xymatrix{
    \CONT_0(\Source)\ar[r]^{\subseteq}
    \ar[d]_{\Id_\Source \times_\Target (\mapl,\Kclass)^*}&
    \CONT_0(\total{\Triv^\Target})
    \ar[d]^{\Id_{\total{\Triv}} \times_\Base (\mapl,\Kclass)^*}
    \ar[r]^{\proj{\Triv^\Target}!}&
    \CONT_0(\Target)
    \ar[d]^{(\mapl,\Kclass)^*}\\
    \CONT_0(\Source\times_\Target\Third) \ar[r]_{\subseteq}&
    \CONT_0(\total{\Triv^\Third}) \ar[r]_{p_\Third!}&
    \CONT_0(\Third)
  }
  \]
  commutes in \(\KK^\Grd\).  Hence \(\KK\bigl(\NM!  \inpro_\Target (\mapl,\Kclass)^*\bigr) = \KK(\NM!)  \otimes_{\CONT_0(\Target)} \KK\bigl((\mapl,\Kclass)^*\bigr)\).  This finishes the proof that \(\Cor\mapsto \KK(\Cor)\) is a functor.

  Now assume that \((\Dual,D,\Theta,\widetilde{\Theta})\) is a \(\K\)\nb-oriented symmetric dual for~\(\Source\), so that Theorem~\ref{the:duality_restricted} provides an isomorphism
  \[
  \GKK_\Grd^*(\Source,\Target)
  \cong \GKK_{\Grd\ltimes\Source}^*
  (\Source, \Dual\times_\Base\Target)
  \cong \RK_{\Grd,\Source}^*(\Dual\times_\Base\Target).
  \]
  The images of \(D\), \(\Theta\) and~\(\widetilde{\Theta}\) in Kasparov theory satisfy analogues of the conditions in Definition~\ref{def:symmetric_dual} because the transformation from \(\GKK\) to \(\KK\) is compatible with composition and exterior products.  For the last two conditions (iv) and~(v), Theorem~\ref{the:geometric_K} and its analogue in Kasparov theory show that there are the same auxiliary data~\(g\) to consider in both theories.  The same computations as in~\S\ref{sec:duality_isomorphisms} yield
  \[
  \KK^\Grd_*\bigl(\CONT_0(\Source), \CONT_0(\Target)\bigr)
  \cong \KK^{\Grd\ltimes\Source}_* \bigl( \CONT_0(\Source),
  \CONT_0(\Dual\times_\Base\Target)\bigr)
  \cong \RK_{\Grd,\Source}^*(\Dual\times_\Base\Target).
  \]
  The last isomorphism is contained in~\cite{Emerson-Meyer:Equivariant_K} (in fact, it is a definition in~\cite{Emerson-Meyer:Equivariant_K}; the results in~\cite{Emerson-Meyer:Equivariant_K} show that this definition agrees with the one used here).  Hence we get the desired isomorphism.
\end{proof}

We leave it to the reader to define \(\KK(\Cor)\) for non-special correspondences directly and to check that \(\KK(\Cor)=\KK(\Cor')\) if~\(\Cor'\) is the special correspondence associated to a correspondence~\(\Cor\).

\begin{corollary}
  \label{cor:compare_to_KK_smooth}
  The natural transformation \(\GKK^*_\Grd(\Source,\Target) \to \KK^\Grd_* \bigl(\CONT_0(\Source), \CONT_0(\Target)\bigr)\) is invertible if~\(\Source\) is normally non-singular.
\end{corollary}

\begin{proof}
  Combine Theorems \ref{the:duality_smooth} and~\ref{the:compare_to_KK}.
\end{proof}

\begin{remark}
  \label{rem:dual_to_dual}
  It is not difficult to check that in fact a symmetric dual in~\(\GKK_\Grd\) arising from a fibrewise stable smooth structure on~\(\Source\) maps to a Kasparov dual in~\(\KK^\Grd\).  This follows from an examination of the proof of \cite{Emerson-Meyer:Dualities}*{Theorem 7.11}, which carries through with no changes from the smooth case.
\end{remark}

\subsection{Smooth correspondences}
\label{sec:smooth_corr}

A \emph{smooth \(\Grd\)-manifold} is a \(\Grd\)\nb-space with a fibrewise smooth structure along the fibres of the anchor map \(\Tot\to \Base\) determined by an atlas for~\(\Tot\) consisting of open sets isomorphic to products \(U\times \R^n\) where \(U\subset \Base\) is an open set.  We also require that the isomorphism intertwines the anchor map and the first coordinate projection, and that the change-of-variables maps are smooth in the vertical direction.

\begin{theorem}[\cite{Emerson-Meyer:Normal_maps}]
  \label{the:normal_map_unique}
  Let \(\Source\) and~\(\Target\) be smooth \(\Grd\)\nb-manifolds, assume that there is a smooth normally non-singular map from~\(\Source\) to the object space~\(\Base\) of~\(\Grd\) and that either~\(\Tvert\Target\) is subtrivial or that all \(\Grd\)\nb-vector bundles on~\(\Source\) are subtrivial.  Then any smooth \(\Grd\)\nb-map from~\(\Source\) to~\(\Target\) is the trace of a smooth normally non-singular \(\Grd\)\nb-map, and two smooth normally non-singular maps from~\(\Source\) to~\(\Target\) are smoothly equivalent if and only if their traces are smoothly homotopic.
\end{theorem}

\begin{remark}
  Theorem~\ref{the:normal_map_unique} fails for non-smooth normally non-singular maps: for a smooth manifold~\(\Source\), there may be several normally non-singular maps \(\Source\to\Source\times\Source\) whose trace is the diagonal embedding.
\end{remark}

Taking into account orientations, one can check that smooth equivalence classes of \(\Coh\)\nb-oriented smooth normally non-singular \(\Grd\)\nb-maps from~\(\Source\) to~\(\Target\) correspond bijectively to pairs \((f,\tau)\) where~\(f\) is the smooth homotopy class of a smooth \(\Grd\)\nb-map from~\(\Source\) to~\(\Target\) and~\(\tau\) is a stable \(\Coh\)\nb-orientation on \([\Tvert\Source] - f^*[\Tvert \Target]\).

\begin{definition}
  \label{def:smooth_correspondence}
  Let \(\Source\) and~\(\Target\) be smooth \(\Grd\)\nb-manifolds.  A \emph{smooth correspondence} from~\(\Source\) to~\(\Target\) is a correspondence \((\Midd,\mapl,\mapr,\Kclass)\) from~\(\Source\) to~\(\Target\) where~\(\Midd\) is a smooth \(\Grd\)\nb-manifold, \(\mapl\) is a fibrewise smooth \(\Grd\)\nb-map, and~\(\mapr\) is a smooth \(\Coh\)\nb-oriented normally non-singular \(\Grd\)\nb-map.  \emph{Smooth bordisms}, \emph{special smooth correspondences}, and \emph{special smooth bordisms} are defined similarly.  We let \emph{smooth equivalence} be the equivalence relation on smooth correspondences generated by smooth equivalence of normally non-singular maps, smooth bordism, and Thom modification by smooth \(\Grd\)\nb-vector bundles.
\end{definition}

The same arguments as in the non-smooth case show that every smooth correspondence is equivalent to a special smooth correspondence and that two special smooth correspondences are smoothly equivalent if and only if they have Thom modifications by trivial \(\Grd\)\nb-vector bundles that are related by a special smooth bordism (see Theorem~\ref{the:special_bordism}).

Let \((\Triv,\Midd,\mapl,\Kclass)\) be a special correspondence from~\(\Source\) to~\(\Target\).  Any \(\Grd\)\nb-vector bundle over~\(\Base\) carries a unique smooth structure for which it is a smooth \(\Grd\)\nb-vector bundle, and this restricts to a unique smooth \(\Grd\)\nb-manifold structure on~\(\Midd\).  Hence a special smooth correspondence does not carry additional structure, it merely has the additional property that the \(\Grd\)\nb-map \(\mapl\colon \Midd\to\Source\) is fibrewise smooth.  The same applies to smooth bordisms.

Since intersection products and exterior products of special smooth correspondences are again special smooth correspondences, the smooth correspondences form a symmetric monoidal category as well.  Theorem~\ref{the:geometric_K} and the duality results in \S\ref{sec:duality_isomorphisms} work for the smooth version of~\(\Bic^*_\Grd\) as well.

\begin{theorem}
  \label{the:smooth_duality}
  Let~\(\Tot\) be a smooth normally non-singular \(\Grd\)\nb-manifold.   Then~\(\Tot\) has a smooth symmetric dual. Furthermore, the smooth and non-smooth versions of \(\Bic^*(\Source,\Target)\) agree in this case, for any smooth \(\Grd\)\nb-manifold~\(\Target\).
\end{theorem}

\begin{proof}
  The constructions in~\S\ref{sec:duality_smooth} produce smooth correspondences if we plug in a smooth normally non-singular map from~\(\Tot\) to~\(\Base\), and the proof of Theorem~\ref{the:duality_smooth} still works in the smooth version of~\(\Bic^*\).  This duality isomorphism allows to identify the two versions of~\(\Bic^*\) as in the proof of Theorem~\ref{the:compare_to_KK}.
\end{proof}

\begin{remark}
  \label{rem:drop_normal_if_smooth}
  If there is a smooth normally non-singular map from~\(\Midd\) to~\(\Base\), then as we have stated above, smooth equivalence classes of smooth normally non-singular maps \(\Midd\to\Target\) correspond bijectively to smooth homotopy classes of smooth maps from~\(\Midd\) to~\(\Target\), and since smooth homotopy is a special case of smooth bordism, we we may drop ``smooth equivalence of normally non-singular maps'' from the definition of smooth equivalence in Definition~\ref{def:smooth_correspondence}.

  The problem with this observation is that we have little control about~\(\Midd\).  It seems that we can only apply it if \emph{all} smooth \(\Grd\)\nb-manifolds~\(\Midd\) admit a smooth normally non-singular \(\Grd\)\nb-map to~\(\Base\).  This holds, for instance, if \(\Grd= G\ltimes\Base\) for a discrete group~\(G\) and a finite-dimensional \(G\)\nb-CW-complex~\(\Base\) with uniformly bounded isotropy groups.  If~\(G\) is, say, a compact group, then we must restrict attention to smooth correspondences \((\Midd,\mapl,\mapr,\Kclass)\) where~\(\Midd\) is a smooth \(G\)\nb-manifold of finite orbit type.

  Summing up, the difference between smooth maps and smooth normally non-singular maps is \emph{usually} insignificant -- but only under some mild technical assumption.  Our theory depends on normally non-singular maps because only this allows to replace a general correspondence by a special one, but it does not depend on smoothness.  This is why we developed our main theory without smoothness assumptions in the main part of this article.
\end{remark}

\section{Outlook and concluding remarks}
\label{sec:outlook}

We have extended an equivariant cohomology theory to a bivariant theory.  In particular, this provides a purely topological counterpart of equivariant Kasparov theory for proper groupoids.  We have used duality isomorphisms to identify the topological and analytic bivariant \(\K\)\nb-theories, and established such duality isomorphisms for smooth \(\Grd\)\nb-manifolds with boundary (under some technical assumptions about equivariant vector bundles).

It is known that any finite-dimensional CW-complex is homotopy equivalent to a smooth manifold with boundary and hence admits a symmetric dual.  We do not know whether a similar result holds equivariantly, say, for simplicial complexes with an action of a finite group.  Anyway, it is desirable for computations to construct symmetric duals for simplicial complexes or even CW-complexes in~\(\Bic^*_\Grd\) directly, without modelling them by smooth manifolds with boundary.  In bivariant Kasparov theory, such a symmetric dual for a simplicial complex is constructed in~\cite{Emerson-Meyer:Euler}, but it involves mildly non-commutative \(\Cst\)\nb-algebras.  It is an open problem to replace this by a purely commutative construction.

An issue that we have neglected here is excision.  Since Kasparov theory satisfies very strong excision results for proper actions, our topological theory will satisfy excision whenever it agrees with Kasparov theory.  But we should not expect good excision results in complete generality: this is one of the points where a lack of enough \(\Grd\)\nb-vector bundles over~\(\Base\) should cause problems.

The first duality isomorphism may be used to define equivariant Euler characteristics and equivariant Lefschetz invariants (see \cites{Emerson-Meyer:Euler, Emerson-Meyer:Dualities}).  Since we have translated it to a purely topological bivariant \(\K\)\nb-theory, we can now compute these invariants geometrically.  We will carry this out for several examples in a forthcoming article.

Although we are mainly interested in \(\K\)\nb-theory and \(\KO\)\nb-theory here, we have allowed more general equivariant cohomology theories in all our constructions.  We expect this to have several applications.

First, since our construction of bivariant theories is functorial with respect to natural transformations of cohomology theories, it should be useful to construct bivariant Chern characters from bivariant equivariant \(\K\)\nb-theory to suitable bivariant Bredon cohomology groups (at least for discrete groups).

Secondly, we hope to define bivariant versions of twisted \(\K\)\nb-theory within our framework.  One approach to this views twisted \(\K\)\nb-theory as a \(\textup{PU}(\Hils)\)-equivariant cohomology theory, where \(\textup{PU}(\Hils)\) denotes the projective unitary group of a separable Hilbert space~\(\Hils\).  A space with a twist datum can be described as a principal \(\textup{PU}(\Hils)\)-bundle.  But we have not yet checked the assumptions we need for our theory to work in this case.

\end{document}